\documentclass[12pt,reqno,a4paper]{amsart} 

\usepackage{fullpage}
\usepackage[english]{babel}
\usepackage{amsmath}
\usepackage{amsfonts,amssymb}
\usepackage{enumerate}
\usepackage{mathrsfs}

\usepackage{amsthm}
\usepackage{tikz}
\usetikzlibrary{arrows}
\usepackage{mathrsfs}
\usepackage{caption}
\usepackage{color}
\usepackage{hyperref} % [hypertex]
\hypersetup{
    colorlinks=true,       
    linkcolor=red,          
    citecolor=blue,        
    filecolor=magenta,      
    urlcolor=cyan           
}

\DeclareMathOperator{\cosech}{cosech}

\theoremstyle{plain}
\newtheorem{theorem}{Theorem}[section]
\newtheorem*{theorem*}{Theorem}
\newtheorem{proposition}[theorem]{Proposition}
\newtheorem{lemma}[theorem]{Lemma}
\newtheorem{corollary}[theorem]{Corollary}
\newtheorem{definition}[theorem]{Definition}

\theoremstyle{definition}
\newtheorem{remark}[theorem]{Remark}
\newtheorem*{remark*}{Remark}

\newtheorem*{notation*}{Notation}
\newcommand{\Til}{\mathcal T}
\newcommand{\R}{\mathbb R}
\newcommand{\HS}{S}
\newcommand{\HVS}{\mathbf S}
\newcommand{\HD}{H_D}
\newcommand{\HN}{H_N}
\newcommand{\Tilh}{\mc T_h}
\renewcommand{\P}{{\mathbf P}}

\newcommand{\dx}{\, d x}

\newcommand\px{\partial_x}

\newcommand{\evaluate}[1]{\widetilde{{#1}}}
\newcommand{\defn}{\mathrel{:=}}

\def\ra{\right\vert}
\def\rA{\right\Vert}
\def\la{\left\vert}
\def\lA{\left\Vert}

\newcommand{\mN}{\mathcal{N}}

\newcommand\mI{\mathcal{I}}

\def\xC{\mathbb{C}}

\def\xR{\mathbb{R}}

\def\xZ{\mathbb{Z}}

\def\mez{\frac{1}{2}}

\newcommand{\mc}{\mathcal}

\newcommand{\mbb}{\mathbb}

\newcommand{\mf}{\mathfrak}

\renewcommand{\Re}{\operatorname{Re}}
\renewcommand{\Im}{\operatorname{Im}}
\newcommand{\Rl}{\mbb R}

\newcommand{\Z}{\mbb Z}

\newcommand{\Strip}{S}

\newcommand{\mfH}{{\mf H}}

\newcommand{\bfB}{\mathbf{B}}

\newcommand{\B}{{\mathbf B}}

\numberwithin{equation}{section}
\begin{document}

\title{A Morawetz inequality for gravity-capillary water waves at low Bond number}
\author[]{Thomas Alazard, Mihaela Ifrim and Daniel Tataru}

\begin{abstract}
This paper is devoted to the 2D gravity-capillary water waves equations in their Hamiltonian formulation, addressing the general question of proving Morawetz inequalities. We continue the analysis initiated in our previous work, where we have established local energy decay 
estimates for gravity waves. Here we add surface tension and prove a stronger estimate
with a local regularity gain, akin to the smoothing effect for dispersive equations. 
Our main result holds globally in time and holds for genuinely nonlinear waves, since we are 
only assuming some very mild uniform Sobolev bounds for the solutions. 
Furthermore, it is uniform both in the infinite depth limit and the zero surface tension limit.

\end{abstract}

\maketitle

\tableofcontents

\section{Introduction}\label{s:Introduction}%%%%%%%%%%%%%%%%%%%%%%%%%%%%%%%%%%%%%%%%%%%%%%%
%%%%%%%%%%%%%%%%%%%%%%%%%%%%%%%%%%%%%%%%%%%%%%%%%%%%%%%%%%%%%%%%%%%%

\subsection{The water-wave equations}
A classical topic in the mathematical theory of hydrodynamics 
concerns the propagation of water waves. 
The problem consists in studying the evolution of the free surface 
separating air from an incompressible perfect fluid, together with the evolution of the velocity field 
inside the fluid domain. 
We assume that the free surface $\Sigma(t)$ is a graph and that 
the fluid domain $\Omega(t)$ has a flat bottom, so that
\begin{align*}
\Omega(t)&=\{ (x,y) \in \xR \times\xR \, \arrowvert\, -h<y<\eta(t,x) \},\\
\Sigma(t) &= \{ (x,y) \in \xR \times\xR \, \arrowvert\, \,  y=\eta(t,x) \},
\end{align*}
where $h$ is the depth and $\eta$ is an unknown (called the free surface elevation). We assume that the velocity field $v\colon \Omega\rightarrow \xR^2$ is irrotational, so that $v=\nabla_{x,y}\phi$ for some potential 
$\phi\colon\Omega\rightarrow \xR$ satisfying 
\begin{equation}\label{t5}
\begin{aligned}
&\Delta_{x,y}\phi=0\quad\text{in }\Omega,\\ 
&\partial_{t} \phi +\mez \la \nabla_{x,y}\phi\ra^2 +P +g y = 0 \quad\text{in }\Omega,\\
&\partial_y\phi =0 \quad \text{on }y=-h,
\end{aligned}
\end{equation}
where 
$P\colon \Omega\rightarrow\xR$ is the pressure, 
$g>0$ is the acceleration of gravity, 
and $\Delta_{x,y}=\partial_x^2+\partial_y^2$. Partial differentiations will often be 
denoted by suffixes so that $\phi_x=\partial_x\phi$ and $\phi_y=\partial_y \phi$.

The 
water-wave problem is described 
by two equations that hold on the free surface: 
firstly an equation describing the time evolution of $\Sigma$:
\begin{equation}\label{eta}
\partial_{t} \eta = \sqrt{1+\eta_x^2}\, \phi_n \arrowvert_{y=\eta}=\phi_y(t,x,\eta(t,x))-
\eta_x(t,x)\phi_x(t,x,\eta(t,x)),
\end{equation}
and secondly an equation for the balance of forces at the free surface:
\begin{equation}\label{pressure}
P\arrowvert_{y=\eta}=-\kappa H(\eta),
\end{equation}
where $\kappa$ is the coefficient of surface tension 
and $H(\eta)$ is the curvature given by
\begin{equation*}
H(\eta)=\px\left( \frac{\eta_x}{\sqrt{1+\eta_x^2}}\right).
\end{equation*}
We begin by recalling the main features of the water wave problem.
\begin{itemize}
\item \textbf{Hamiltonian system. } 
Since $\phi$ is harmonic function with Neumann boundary condition on the 
bottom, it is fully determined by 
its trace on $\Sigma$. Set 
$$
\psi(t,x)=\phi(t,x,\eta(t,x)).
$$
Zakharov discovered that $\eta$ and $\psi$ are canonical variables. 
Namely, he gave the following Hamiltonian formulation 
of the water-wave equations (\cite{Zakharov1968,Zakharov1998}):
$$
\frac{\partial \eta}{\partial t}=\frac{\delta \mathcal{H}}{\delta \psi},\quad 
\frac{\partial \psi}{\partial t}=-\frac{\delta \mathcal{H}}{\delta \eta},
$$
where $\mathcal{H}$ is the energy, which reads
\begin{equation}\label{Hamiltonian}
\mathcal{H}=\frac{g}{2}\int_\R \eta^2\dx+\kappa \int_\R \left(\sqrt{1+\eta_x^2}-1\right)\dx
+\mez\iint_{\Omega(t)}\la \nabla_{x,y}\phi \ra^2\, dydx.
\end{equation}
The Hamiltonian is the sum of the gravitational potential energy, 
a surface energy due to stretching of the surface and the kinetic energy. 
One can give more explicit evolution equations 
by introducing the 
Dirichlet to Neumann operator associated to the fluid domain $\Omega(t)$, defined by 
$$
G(\eta)\psi=\sqrt{1+\eta_x^2}\, \phi_{n \arrowvert_{y=\eta}}=(\phi_y-\eta_x \phi_x)_{\arrowvert_{y=\eta}}.
$$
Then (see \cite{LannesLivre}), 
with the above notations, the water-wave system reads
\begin{equation}\label{systemT}
\left\{
\begin{aligned}
&\partial_t \eta=G(\eta)\psi \\
&\partial_t \psi+g\eta +\mez \psi_x^2-\mez \frac{(G(\eta)\psi+\eta_x\psi_x)^2}{1+\eta_x^2}-\kappa H(\eta)=0.
\end{aligned}
\right.
\end{equation}
Of course the Hamiltonian is conserved along the flow. Another conservation law that is essential in this paper is the conservation of the horizontal momentum,
\begin{equation} \label{momentum}
\mathcal M = \int_\R \eta \psi_x \, dx.  
\end{equation}
From a Hamiltonian perspective, this can be seen as arising via Noether's theorem as
the generator for the horizontal translations, which commute with the water wave flow.

\item \textbf{Scaling invariance.} 
Another symmetry is given by the scaling invariance which holds 
in the infinite depth case (that is when $h=\infty$)  when either $g = 0$ or $\kappa = 0$.

If $\kappa = 0$  and  $\psi$ and $\eta$ are solutions of the gravity water waves equations \eqref{systemT}, then $\psi_\lambda$ and $\eta_\lambda$ defined by
\[
\psi_\lambda(t,x)=\lambda^{-3/2} \psi (\sqrt{\lambda}t,\lambda x),\quad 
\eta_\lambda(t,x)=\lambda^{-1} \eta(\sqrt{\lambda} t,\lambda x),
\]
solve the same system of equations. The (homogeneous) 
Sobolev spaces invariant by this scaling 
correspond to $\eta$  in $\dot{H}^{3/2}(\xR)$ and $\psi$ in $\dot{H}^2(\xR)$. 

On the other hand if $g = 0$ and $\psi$ and $\eta$ are solutions of
the capillary water waves equations \eqref{systemT}, then
$\psi_\lambda$ and $\eta_\lambda$ defined by
\[
\psi_\lambda(t,x)=\lambda^{-\frac12} \psi (\lambda^\frac32 t,\lambda x),\quad 
\eta_\lambda(t,x)=\lambda^{-1} \eta(\lambda^\frac32 t,\lambda x),
\]
solve the same system of equations. The (homogeneous) Sobolev spaces invariant by this scaling correspond to $\eta$  in $\dot{H}^{3/2}(\xR)$ and $\psi$ in $\dot{H}^1(\xR)$. 

\item 
This is a \textbf{quasi-linear} system of \textbf{nonlocal} equations. 
As a result, even the study of the Cauchy problem for smooth data is highly nontrivial. 
The literature on this topic is extensive by now, 
starting with the works of Nalimov~\cite{Nalimov} and
Yosihara~\cite{Yosihara}, who proved existence and uniqueness in Sobolev spaces under a smallness assumption. Without a smallness assumptions on the data, the well-posedness of the Cauchy problem     
was first proved by Wu~\cite{Wu97, Wu99} without surface tension and 
by Beyer-G\"unther in~\cite{BG} in the case with surface tension. 
Several extensions of these results were obtained 
by various methods and many authors. We begin by quoting recent results for gravity-capillary waves. 
For the local in time Cauchy problem 
we refer to \cite{Agrawal-LWP,ABZ1,AmMa,CMSW,CS,IguchiCPDE,LannesKelvin,LannesLivre,MingWang-LWP,RoTz,Schweizer,ScWa,ShZe}, 
see also~\cite{ITc} and \cite{IoPuc,X-Wang-capillary} for global existence results 
for small enough initial data which are localized, 
and \cite{CCFGS-surface-tension,FeIoLi-Duke-2016} for results about splash singularities 
for large enough initial data. 
Let us recall that the Cauchy problem for the 
gravity-capillary water-wave equations is locally 
well-posed in suitable function spaces which are $3/2$-derivative more 
regular than the scaling invariance, e.g.\ when initially
\[
\eta \in H^{s+\mez}(\xR), \qquad \psi \in H^{s}(\xR), \qquad s >\frac 5 2.
\]
Actually, some better results hold using Strichartz estimates (see \cite{dPN1,dPN2,Nguyen-AIHPANL-2017}). 
There are also many recent results for the equations 
without surface tension, and we refer 
the reader to the papers~\cite{Ai1,Ai2,AD-global,CCFGGS-arxiv,CS2,dePoyferre-ARMA,GMS,GuTi,H-GIT,HIT,IT,IoPu,Wu09,Wu10}.

\item \textbf{Dispersive equation.} 
Consider the linearized water-wave equations:
 $$
\left\{
 \begin{aligned}
& \partial_t \eta=G(0)\psi=\la D\ra\tanh(h\la D\ra) \psi, \qquad (\la D\ra=\sqrt{-\partial_x^2})\\
& \partial_t \psi+g\eta-\kappa \partial_{x}^2\eta=0.
\end{aligned}
\right.
$$
Then the dispersion relationship reads 
$\omega^2= k \tanh(h k) (g+\kappa k^2)$, which 
shows that water waves are dispersive waves. 
The dispersive properties have been studied for many different problems, 
including the global in time existence results alluded to before and 
also various problems about Strichartz estimates, 
local smoothing effect, control theory 
or the study of solitary waves (see e.g.\ \cite{Ai1,Ai2,ABZ1,CHS,dPN2,dPN1,IT2018solitary,Zhu1,Zhu2}).

\end{itemize}

\subsection{Morawetz estimates}
Despite intensive researches on dispersive or Hamiltonian equations, 
it is fair to say that many natural questions concerning 
the dynamics of water waves are mostly open. 
Among these, we have initiated in our previous work \cite{AIT} 
the study of Morawetz estimates for water waves. 
In this paragraph, we introduce this problem  within the general framework of Hamiltonian equations. 

Consider a Hamiltonian system of the form 
$$
\frac{\partial \eta}{\partial t}=\frac{\delta \mathcal{H}}{\delta \psi},\quad 
\frac{\partial \psi}{\partial t}=-\frac{\delta \mathcal{H}}{\delta \psi}\quad\text{with}\quad 
\mathcal{H}=\int e\dx,
$$
where the density of energy 
$e$ depends only on $\eta$ and $\psi$. 
This setting includes the water-wave equations, the 
Klein-Gordon equation, the Schr\"odinger equation, the Korteveg-de-Vries equation, etc. 
For the sake of simplicity, let us compare the following linear equations: 
\begin{itemize} 
\item (GWW) the gravity water-wave equation $\partial_t^2 u+\la D_x\ra u=0$;
\item (CWW) the capillary water-wave equation $\partial_t^2 u+\la D_x\ra^3 u=0$;
\item (KG) the Klein-Gordon equation $\square u+u=0$;
\item (S) the Schr\"odinger equation $i\partial_t u+\Delta u=0$. 
\end{itemize}
All of them  can be written in the above Hamiltonian form with
\begin{align*} 
&e_{\text{GWW}}=\eta^2+ (\la D_x\ra^{1/2} \psi)^2,\\
&e_{\text{CWW}}=(\la D_x\ra\eta)^2+ (\la D_x\ra^{1/2} \psi)^2,\\
&e_{\text{KG}}=(\langle D_x\rangle^{1/2}\eta)^2+ (\langle D_x\rangle^{1/2}\psi)^2,\\
& e_{\text{S}}=(\la D_x\ra \eta)^2+(\la D_x\ra \psi)^2.
\end{align*}
For such Hamiltonian systems, one can deduce from the Noether's theorem and 
symmetries of the equations some conserved quantities. 
For instance, the invariance by translation in time implies that 
the energy $\mathcal{H}$ is conserved, while the invariance by spatial translation implies that 
the momentum $\mathcal{M}=\int \eta\psi_x \dx$ is conserved. 
Then, a Morawetz estimate is an estimate of the local energy in terms of a quantity which scales 
like the momentum. 
More precisely, given a time $T>0$ and a compactly supported function $\chi=\chi(x)\ge 0$, one seeks an estimate 
of the form
$$
\int_0^T \int_\xR \chi (x)e(t,x) \,dx dt \le C(T) \sup_{t\in [0,T]}\lA \eta(t)\rA_{H^s(\xR)}
\lA \psi(t)\rA_{H^{1-s}(\xR)},
$$
for some real number $s$ chosen in a balanced 
way depending on the given equation. If the constant 
$C(T)$ does not depend on time $T$, then 
we say that the estimate is global in time. The study of the latter estimates was 
introduced in Morawetz's paper \cite{Morawetz1968} in the context of the nonlinear Klein-Gordon 
equation. 

The study of Morawetz estimates is interesting for both linear and nonlinear equations. We begin by discussing the linear phenomena. 
For the Klein-Gordon equation, 
the local energy density $e_{\text{KG}}$ measures the same 
regularity as the momentum density $\eta\psi_x$ so that only the global 
in time estimate is interesting. The latter result expresses the fact that 
the localized energy is globally integrable in time, 
and hence has been called a {\em local energy decay} result. 
Morawetz estimates have also been proved for the Schr\"odinger equation. Here 
the natural energy density $e_{\text{S}}$ measures a higher
regularity than the momentum density $\eta \psi_x$; for this reason the result 
is meaningful even locally in time and the resulting (local in time) Morawetz estimates 
have been originally called \emph{local smoothing estimates}, see~\cite{CS-ls,Vega,PlVe}. 
The same phenomena appears also for the KdV equation. 
In fact, the general study of Morawetz estimates has had a long 
history, which is too extensive to try to describe here. 
For further recent references we refer the reader to \cite{MT} for the wave equation, \cite{MMT} for the Schr\"odinger equation, 
\cite{OR} for fractional dispersive equations. These estimates are very robust and also hold for nonlinear problems, which make them useful in the  study of 
the Cauchy problem for nonlinear equations (see e.g.\, \cite{CKSTT,KPVinvent,PlVe}).

We are now ready to discuss Morawetz estimates for water waves. 
One of our motivations to initiate their study in \cite{AIT} 
is that this problem exhibits interesting new features. 
Firstly, since the problem is nonlocal, it is 
already difficult to obtain a global in time estimate for the linearized equations. 
The second and key observation is that, even if the equation is \emph {quasilinear}, one can prove such global in time estimates for the nonlinear equations, assuming some very mild smallness assumption on the solution. 

Given a compactly supported bump function $\chi=\chi(x)$, 
we want to estimate the local energy
\[
 \int_0^T \!\!\!\int_\xR\chi(x-x_0) ( g \eta^2 + \kappa \eta_x^2) \, dx dt
+ \int_0^T \!\!\!\int_\xR\int_{-h}^{\eta(t,x)} \chi(x-x_0)
\la \nabla_{x,y}\phi \ra^2 \, dydx dt,
\]
uniformly in time $T$ and space location $x_0$, assuming only a uniform bound 
on the size of the solutions. In our previous paper \cite{AIT}, 
we have studied this problem for gravity water waves (that is for $\kappa=0$). There the  momentum is not controlled by the Hamiltonian energy, and as a result 
we have the opposite phenomena to local smoothing, namely a loss 
of $1/4$ derivative in the local energy. This brings 
substantial difficulties in the low frequency analysis, 
in particular in order to prove a global in time estimate. 
In addition, we also took into account the effect of the bottom, 
which generates an extra difficulty in the analysis of 
the low frequency component. 

In this article we assume that $\kappa\ge 0$ so that 
one can both study gravity water waves and gravity-capillary water waves 
(for $\kappa>0$). Furthermore, 
we seek to prove Morawetz estimates uniformly with respect 
to surface tension $\kappa$ as $\kappa \to 0$, and also uniformly with respect to the depth $h$ as $h \to \infty$. In this context we will impose a smallness condition on the Bond number.

\subsection{Function spaces} 
As explained above, the goal of the present 
paper is to be able to take into account surface tension effects
i.e. the high frequency local smoothing, while still 
allowing for the presence of a bottom which yields substantial difficulties at 
low frequencies. In addition, one key point is that 
our main result holds provided that some mild Sobolev norms
of the solutions remain small enough uniformly in time. Here
we introduce the functional setting which will allow us to do so. 

Precisely, in this subsection we introduce three spaces: 
a space $E^0$ associated to the energy, 
a space $E^{\frac14}$ associated to 
the momentum, and a uniform in time control norm $\lA \cdot\rA_{X^\kappa}$ 
which respects the scaling invariance. 

The above energy $\mathcal{H}$ (Hamiltonian) corresponds to the energy space for $(\eta,\psi)$,
\[
E^0 = \left( g^{-\frac12} L^2(\xR) \cap \kappa^{-\frac12} \dot H^1(\R) \right) \times \dot{H}^{\frac12}_h(\xR),
\]
with the depth dependent $H^{\frac12}_h(\xR)$ space defined as
\[
\dot{H}^{\frac12}_h(\xR) = \dot H^\frac12(\xR) + h^{-\frac12} \dot H^1(\xR).
\]
We already observe in here the two interesting frequency thresholds in this problem, namely 
$h^{-1}$, determined by the depth, and $\sqrt{g/\kappa}$, determined by the balance between gravity
and capillarity. The dimensionless connection between these two scales is described by the Bond number,
\[
\B = \frac{\kappa}{gh^2}.
\]
A key assumption in the present work is that 
\[
\B \ll 1.
\]
This is further discussed in the comments  following our main result.

Similarly, in order to measure the momentum, we use the space
$E^{\frac14}$, which is the $h$-adapted linear $H^{\frac14}$-type norm for $(\eta,\psi)$ (which corresponds to the momentum),
\begin{equation}\label{e14}
\begin{aligned}
&E^{\frac14} := \left( g^{-\frac14} H^{\frac14}_h(\xR) \cap \kappa^{-\frac14} H^{\frac34}_h(\xR)\right)  \times \left( g^\frac14   \dot H^{\frac34}_h(\xR) +  \kappa ^\frac14   \dot H^{\frac14}_h(\xR) \right), \\
\end{aligned}
\end{equation}
with
\[
H^{s}_h(\xR) := \dot H^s(\xR) \cap  h^{s} L^2(\xR), \qquad   \dot H^{s}_h(\xR) = \dot{H}^s (\xR)+  h^{s-1} \dot H^1(\xR),
\]
so that in particular we have 
\[
|\mathcal M| \lesssim  \|(\eta,\psi)\|_{E^\frac14}^2.
\]
We remark that there is some freedom here in choosing the space $E^{\frac14}$; the one above is 
not as in the  previous paper \cite{AIT}, adapted to 
gravity waves, but is instead  changed above the frequency threshold $\lambda_0 = \sqrt{g/\kappa}$ and adapted to capillary waves instead at high frequency.

For our uniform a-priori bounds for the solutions, we begin by recalling the set-up in \cite{AIT}
for the pure gravity waves. There we were able  to use 
a scale invariant norm, which corresponds to the following Sobolev bounds:
\[
\eta \in H^\frac32_h(\xR), \qquad \nabla \phi\arrowvert_{y=\eta} \in g^\frac12 H^1_h(\xR) 
\cap \kappa^{\frac12} L^2(\xR).
\]
Based on this, we have introduced the homogeneous norm $X_0$ defined by
\[
X_0 := L^\infty_t H^{\frac{3}{2}}_{h} \times g^{-\mez}L^\infty_t H^1_h,
\]
and used it to define the  uniform control norm $X$ by
\[
\|(\eta,\psi)\|_{X} := \lA P_{\leq h^{-1}}
(\eta,\nabla \psi)\rA_{X_0}+
\sum_{\lambda > h^{-1}} 
\|P_\lambda (\eta,\psi)\|_{X_0} .
\]
Here we use a standard Littlewood-Paley decomposition beginning at frequency $1/h$,
\[
1 = P_{<1/h} + \sum_{1/h < \lambda \in 2^\Z} P_\lambda.
\]

The  uniform control norm we use in this paper, denoted by $X^\kappa$, 
matches the above scenario at low frequency, but requires some strengthening at high
frequency. The threshold frequency in this context is
\[
\lambda_0 = \sqrt{g/\kappa} \gg 1,
\]
and describes the transition from  gravity to capillary waves. Then we will complement the $X$ 
bound with  a stronger bound  for $\eta$ in the higher frequency range, setting 
\[
X_1 :=  \left(\frac{g}{\kappa}\right)^\frac14   L^\infty_t H^{2}.
\]
One can verify that this  matches the first component of $X_0$ exactly at frequency $\lambda_0$.
Then our uniform control norm $X^{\kappa}$ will be
\[
\|(\eta,\psi)\|_{X^\kappa} := \lA
(\eta,\nabla \psi)\rA_{X}+\| \eta\|_{X_1}.
\]

Based on the expression \eqref{Hamiltonian} for the
energy, we introduce the following notations for the local energy. Fix an 
arbitrary compactly supported nonnegative function $\chi$. 
Then,  the local energy  centered around a point $x_0$ is
\[
\begin{split}
\| (\eta,\psi)\|_{LE^\kappa_{x_0}}^2 :=   & \ g \int_0^T \int_\xR\chi(x-x_0) \eta^2\,dx dt + \kappa \int_0^T \int_\xR \chi(x-x_0)\eta_{x}^2\dx dt 
\\ & \ +\int_0^T \int_\xR\int_{-h}^{\eta(t,x)} \chi(x-x_0) \la \nabla_{x,y}\phi \ra^2\, dy dx dt.
\end{split}
\]
It is also of interest to take the supremum over $x_0$,
\begin{equation}\label{defiLEkappa}
\| (\eta,\psi)\|_{LE^\kappa}^2 :=  \sup_{x_0\in \xR} \| (\eta,\psi)\|_{LE_{x_0}}^2.
\end{equation}

\subsection{The main result}
Our main Morawetz estimate  for gravity-capillary water waves is as follows:

\begin{theorem}[Local energy decay for gravity-capillary waves]\label{ThmG1}
  Let $s>5/2$ and $C > 0$. There exist $\epsilon_0$ and $C_{0}$ such
  that the following result holds.  For all $T\in (0,+\infty)$, all
  $h\in [1,+\infty)$, all $g\in (0,+\infty)$, all $\kappa\in (0,+\infty)$
  with $\kappa \ll g$ and all solutions $(\eta,\psi)\in
  C^0([0,T];H^{s}(\R)\times H^s(\R))$ of the water-wave system
  \eqref{systemT} satisfying
\begin{equation}\label{uniform}
\| (\eta,\psi)\|_{X^\kappa} \leq \epsilon_0
\end{equation}
the following estimate holds
\begin{equation}\label{nMg}
\begin{aligned}
\| (\eta,\psi)\|_{LE^\kappa}^2  \leq C_0 ( \| (\eta, \psi) (0) \|^2_{E^{\frac14}}+ \| (\eta, \psi) (T) \|^2_{E^{\frac14}}).
\end{aligned}
\end{equation}
\end{theorem}

We continue with several remarks concerning the choices of parameters/norms in the theorem.

\begin{remark}[Uniformity]
One key feature of our result is 
that it is global in time (uniform in $T$) and uniform in $h \geq 1$ and $\kappa \ll g$. In particular our estimate is uniform in the infinite depth limit and the zero surface tension limit.
\end{remark}

\begin{remark}[Time scaling]
Another feature of our result is that the statement of Theorem ~\ref{ThmG1} is invariant with respect to the following scaling law  (time associated scaling)  
\[
\begin{aligned}
(\eta (t,x), \psi(t,x)) &\rightarrow (\eta (\lambda t,x), \lambda \psi(\lambda t,x))\\
(g,\kappa,h) &\rightarrow (\lambda^2 g, \lambda^2 \kappa, h).
\end{aligned}
\]
This implies that the value of $g$ and $\kappa$ separately are not
important but their ratio is.  By scaling one could simply set 
$g = 1$ in all the proofs. We do not do that in order to improve the
readability of the article.
\end{remark}

\begin{remark}[Window size]
Once we have the local energy decay bounds for a unit window size, we also trivially have them 
for any higher window size $R > 1$, with a constant of $R$ on the right in the estimate. 
\end{remark}

\begin{remark} [Spatial scaling]
One can also rescale the spatial coordinates $x \to \lambda x$, and correspondingly
\[
\begin{aligned}
(\eta (t,x), \psi(t,x)) &\rightarrow (  \lambda^{-1}\eta (t, \lambda x), \lambda^{-\frac32} \lambda \psi(t, \lambda x))\\
(g,\kappa,h,R) &\rightarrow (g, \lambda{-2} \kappa, \lambda^{-1} h, \lambda^{-1} R).
\end{aligned}
\]
Then combining this with respect to the previous remark, we can restate our hypothesis $\kappa/g \ll 1 \leq h^2$ as a constraint on the Bond number, $\B \ll 1$, together with a  window size restriction as $\kappa/g \lesssim R^2 \lesssim h^2$.
\end{remark}

\begin{remark}
As already explained, the uniform control norms in  \eqref{uniform} are below the current local well-posedness threshold for this problem, and are instead closer to  one might view as the 
critical, scale invariant norms for this problem. The dependence on $h$ 
is natural as spatial scaling will also alter the depth $h$. In the infinite depth limit one recovers exactly the homogeneous Sobolev norms at low frequency.
We also note that, by Sobolev embeddings, our smallness assumption
guarantees that 
\[
|\eta| \lesssim \epsilon_0 h, \qquad |\eta_x | \lesssim \epsilon_0.
\]
\end{remark}

 Our previous work \cite{AIT} on gravity waves followed 
the same principles as in Morawetz's original paper (\cite{Morawetz1968}), proving
the results using the multiplier method, based on the momentum conservation law. One difficulty
we encountered there was due to the nonlocality of the problem, which made it 
far from obvious what is the ``correct" momentum density. Our solution in \cite{AIT} was to 
work in parallel with two distinct momentum densities.

Some of the difficulties in \cite{AIT} were connected to low frequencies, 
 due both to the fact that the equations are nonlocal, and that they have quadratic nonlinearities.
 A key to defeating these difficulties was to shift some of the analysis 
 from Eulerian coordinates and the holomorphic coordinates; this is because the latter provide a better setting to understand the fine bilinear and multilinear structure of the equations.

Here we are able to reuse the low frequency part of the analysis from \cite{AIT}. On the other 
hand, we instead encounter additional high frequency issues arising from the surface tension contributions.
As it turns out, these are also best dealt with in the set-up of holomorphic coordinates.

\subsection{Plan of the paper} 
The key idea in our proof of the Morawetz estimate is to think of the energy as the flux for the momentum. Unfortunately this is not as simple as it sounds, as the momentum density is not a uniquely defined object, and the ``obvious" choice does not yield the full energy as its flux, not even at leading order. To address this matter, in the next section, we review density flux pairs for the momentum. The standard density $\eta\psi_x$ implicit in \eqref{momentum} only allows one to control the  local potential energy, while for the local kinetic energy we introduce
an alternate density and the associated flux.

The two density-flux relations for the momentum are exploited in Section~\ref{s:nonlinear}. 
There the proof of the Morawetz inequality is reduced to three main technical lemmas.
However, proving these lemmas in the standard Eulerian setting seems nearly impossible; instead,
our strategy is to first switch to holomorphic coordinates.

We proceed as follows. In Section~\ref{ss:laplace} we review the holomorphic (conformal) coordinates
and relate the two sets of variables between Eulerian and holomorphic formulation. In particular
we show that the fixed energy type norms in the paper admit equivalent formulations in the two settings. In this we largely follow \cite{HIT}, \cite{H-GIT} and \cite{AIT}, though several 
new results are also needed. In Section~\ref{s:switch} we discuss the correspondence between the 
local energy norms in the two settings, as well as several other related matters.

The aim of the last section of the paper is to prove the three key main technical lemmas. 
This is done in two steps. First we obtain an equivalent formulation in the holomorphic setting, and then we use multilinear analysis methods to prove the desired estimates.

\subsection*{Acknowledgements} 
The first author was partially supported by the SingFlows project, 
grant ANR-18-CE40-0027 
of the French National Research Agency (ANR).
The second author was partially was supported by a Luce Assistant Professorship, by the Sloan Foundation, and by an NSF CAREER grant DMS-1845037. The third author was partially supported by the NSF grant DMS-1800294 
as well as by a Simons Investigator grant from the  Simons Foundation.

%%%%%%%%%%%%%%%%%%%%%%%%%%%%%%%%%%%%%%%%%%%%%%%%%%%%%%%%%%%%%%%%%%%%
\section{Conservation of momentum and local conservation laws}\label{s:Cl}%%%%%%%%%%%%%%%%%%%%%%%%%%%%%%
%%%%%%%%%%%%%%%%%%%%%%%%%%%%%%%%%%%%%%%%%%%%%%%%%%%%%%%%%%%%%%%%%%%%

A classical result is that the momentum, which is the following quantity,
\[
\mathcal{M}(t)=\int_\xR\int_{-h}^{\eta(t,x)}\phi_x (t,x,y)\, dy dx,
\]
is a conserved quantity:
$$
\frac{d}{dt} \mathcal{M}=0.
$$
This comes from the invariance with respect to horizontal translation (see Benjamin
and Olver \cite{BO} for a systematic study of the symmetries of
the water-wave equations). 
We exploit later the conservation of the momentum through the use of 
density/flux pairs $(I,S)$. By definition, these are pairs such that 
\begin{equation}\label{density}
\mathcal{M}=\int I \, dx ,
\end{equation}
and such that one has the conservation law
\begin{equation}\label{cl}
\partial_t I+\partial_x S=0.
\end{equation}
One can use these pairs by means of the multiplier method of Morawetz. 
Consider a function $m=m(x)$ which is positive and increasing. 
Multiplying the identity  \eqref{cl} by $m=m(x)$ and integrating 
over $[0,T]\times \xR$ then yields (integrating by parts)
\[
\iint_{[0,T]\times \xR} S(t,x) m_x \, dxdt =\int_\xR m(x)I(T,x)\, dx-\int_\xR m(x)I(0,x)\, dx.
\]
The key point is that, since the slope $m_x$ is nonnegative, the later identity is favorable provided 
that $S$ is non-negative.

There are three pairs $(I,S)$ that  play a role in our work. 
These pairs have  already 
been discussed in \cite[\S 2]{AIT}. Here we keep the same densities;
however, the associated fluxes will acquire an extra term  due to 
the surface tension term in the equations.

\begin{lemma}
The expression 
\[
I_1(t,x)=\int_{-h}^{\eta(t,x)}  \phi_x(t,x,y) \, dy,
\]
is a density for the momentum, with associated density flux
\[
S_1(t,x)\defn -\int_{-h}^{\eta(t,x)}\partial_t \phi\, dy -\frac{g}{2}\eta^2
+ \kappa\left(1-\frac{1}{\sqrt{1+\eta_x^2}}\right)
+\mez \int_{-h}^{\eta(t,x)}(\phi_x^2-\phi_y^2)\, dy.
\]
\end{lemma}
\begin{proof}
With minor changes this repeats the computation in \cite{AIT}. 
Given 
a function $f=f(t,x,y)$, we use the notation 
$\evaluate{f}$ to denote the function 
$$
\evaluate{f}(t,x)=f(t,x,\eta(t,x)).
$$
Then we have
$$
\partial_t I_1=\partial_t \int_{-h}^{\eta}\phi_x\, dy=(\partial_t \eta)\evaluate{\phi_x}+\int_{-h}^\eta \partial_t \phi_x \, dy.
$$
It follows from the kinematic equation for $\eta$ and 
the Bernoulli equation for $\phi$ that 
$$
\partial_t I_1=(\evaluate{\phi_y}-\eta_x\evaluate{\phi_x})\evaluate{\phi_x}
-\int_{-h}^{\eta}\partial_x \left(\mez \la\nabla_{x,y}\phi\ra^2+ P\right)\, dy,
$$
so
$$
\partial_t I_1=\evaluate{\phi_y}\evaluate{\phi_x}+\mez \eta_x\evaluate{\phi_y^2}
-\mez\eta_x\evaluate{\phi_x}^2+\eta_x \evaluate{P}
-\partial_x \int_{-h}^{\eta}\left(\mez \la\nabla_{x,y}\phi\ra^2+ P\right) \, dy.
$$
Using again the Bernoulli equation and using the 
equation for the pressure at the free surface), 
we find that
$$
\partial_t I_1=\evaluate{\phi_y}\evaluate{\phi_x}+\mez \eta_x\evaluate{\phi_y^2}
-\mez\eta_x\evaluate{\phi_x}^2 %\blue{
-\kappa \eta_x \partial_x\left(\frac{\eta_x}{\sqrt{1+\eta_x^2}}\right)%}
+\partial_x \int_{-h}^{\eta}\left(\partial_t \phi+gy\right) \, dy.
$$
%\blue{
Since%We first compute
$$
-\kappa \eta_x \partial_x\left(\frac{\eta_x}{\sqrt{1+\eta_x^2}}\right)
=\kappa\partial_x \frac{1}{\sqrt{1+\eta_x^2}}=\kappa \partial_x\left(\frac{1}{\sqrt{1+\eta_x^2}}-1\right),
$$%}
and since $\partial_x \int_{-h}^{\eta}gy\, dy=\partial_x(g\eta^2/2)$, to complete the proof, it is sufficient to verify that
$$
\evaluate{\phi_y}\evaluate{\phi_x}+\mez \eta_x\evaluate{\phi_y^2}
-\mez\eta_x\evaluate{\phi_x}^2=\mez\partial_x \int_{-h}^{\eta}(\phi_y^2-\phi_x^2)\, dy.
$$
This in turn follows from the fact that
$$
\int_{-h}^\eta(\phi_x\phi_{yx}-\phi_x\phi_{xx})\, dy=
\int_{-h}^\eta(\phi_x\phi_{yx}+\phi_x\phi_{yy})\, dy=
\int_{-h}^\eta\partial_y (\phi_x\phi_y)\, dy=\evaluate{\phi_x}\evaluate{\phi_y},
$$
where we used $\phi_{xx}=-\phi_{yy}$ and the solid wall boundary condition $\phi_y(t,x,-h)=0$.
\end{proof}

The above density-flux pair will not 
be directly useful because it has a linear component in it. 
However, we will use it as a springboard 
for the next two density-flux pairs.

\begin{lemma}\label{l:I2}
The expression 
\[
I_2(t,x)=\eta(t,x)\psi_x(t,x)
\]
is a density for the momentum, with associated density flux
\[
S_2(t,x)\defn -\eta\psi_t-\frac{g}{2}\eta^2+ \kappa\left(1- \frac{1}{\sqrt{1+\eta_x^2}}\right)+
\mez \int_{-h}^{\eta(t,x)}(\phi_x^2-\phi_y^2)\, dy.
\]
\end{lemma}
\begin{proof}
The proof is identical to the one of Lemma~2.3~in~\cite{AIT}.
\end{proof}

To define the third pair we recall from \cite{AIT} two auxiliary functions defined inside the fluid domain. 
Firstly we introduce the stream function~$q$, 
which is the harmonic conjugate of $\phi$:
\begin{equation}\label{defi:qconj}
\left\{
\begin{aligned}
& q_x = -\phi_y , \quad \text{in }-h<y<\eta(t,x),\\
& q_y = \phi_x, \qquad \text{in }-h<y<\eta(t,x),\\
& q(t,x,-h) = 0.
\end{aligned}
\right.
\end{equation}
We also introduce the harmonic extension $\theta$ of $\eta$ with Dirichlet boundary condition on the bottom:
\begin{equation}\label{defi:thetainitial}
\left\{
\begin{aligned}
& \Delta_{x,y}\theta=0\quad \text{in }-h<y<\eta(t,x),\\
&\theta(t,x,\eta(t,x))=\eta(t,x),\\
&\theta(t,x,-h) = 0.
\end{aligned}
\right.
\end{equation}
Now the following lemma contains the key density/flux pair for the momentum. 
\begin{lemma}\label{l:I3}
The expression 
\[
I_3(t,x)= \int_{-h}^\eta 
\nabla \theta(t,x,y) \cdot\nabla q(t,x,y) \, dy
\]
is a density for the momentum, with associated density flux
\[
S_3(t,x)\defn  -\frac{g}{2}\eta^2 - \int_{-h}^{\eta(t,x)} \theta_y \phi_t \, dy   
 + \kappa\left(1 - \frac{1}{\sqrt{1+\eta_x^2}}\right)
+  \int_{-h}^{\eta(t,x)} \Big(\mez (\phi_x^2 - \phi_y^2)+ \theta_t \phi_y \Big)\, dy.
\]
\end{lemma}
\begin{proof}
The proof is identical to the one of Lemma~2.4~in~\cite{AIT}.
\end{proof}

In our analysis we will not only need the evolution equations restricted to the free boundary,
but also the evolution of $\theta$ and $\phi$ within the fluid domain.
To describe that, we introduce the operators
$\HD$ and $\HN$, which act 
on functions on the free surface 
$\{y=\eta(t,x)\}$ and produce 
their harmonic extension inside 
the fluid domain with zero Dirichlet, respectively Neumann
boundary condition\footnote{These two operators coicide in the infinite depth setting.} on the bottom $\{y=-h\}$. 
With these notations, we have
\[
\theta = \HD(\eta), \qquad \phi= \HN(\psi).
\]
Recall that, given a function $f=f(t,x,y)$, we set 
$\evaluate{f}(t,x):=f(t,x,\eta(t,x))$.
\begin{lemma}
The function $\phi_t$ is harmonic within the fluid domain, 
with Neumann boundary condition on the bottom, and satisfies 
\begin{equation}\label{phit}
\phi_t = -g \HN(\eta) -  \HN\left(\evaluate{|\nabla \phi|^2}\right) + \kappa \HN( \partial_x(\eta_x/\sqrt{1+\eta_x^2})      ).
\end{equation}
The function $\theta_t$ is harmonic within the fluid domain, 
with Dirichlet boundary condition on the bottom, and satisfies
\begin{equation}\label{thetat}
\theta_t = \phi_y -  \HD\left(\evaluate{\nabla \theta \cdot\nabla \phi}\right).
\end{equation}
\end{lemma}
\begin{proof}
The equation for $\phi$ follows directly from the Bernoulli equation. The equation for $\theta$
is as in \cite[Lemma~2.5]{AIT}.
\end{proof}

%%%%%%%%%%%%%%%%%%%%%%%%%%%%%%%%%%%%%%%%%%%%%%%%%%%%%%%%%%%%%%%%%%%%
%%%%%%%%%%%%%%%%%%%%%%%%%%%%%%%%%%%%%%%%%%%%%%%%%%%%%%%%%%%%%%%%%%%%
%%%%%%%%%%%%%%%%%%%%%%%%%%%%%%%%%%%%%%%%%%%%%%%%%%%%%%%%%%%%%%%%%%%%
%%%%%%%%%%%%%%%%%%%%%%%%%%%%%%%%%%%%%%%%%%%%%%%%%%%%%%%%%%%%%%%%%%%%
%%%%%%%%%%%%%%%%%%%%%%%%%%%%%%%%%%%%%%%%%%%%%%%%%%%%%%%%%%%%%%%%%%%%
%%%%%%%%%%%%%%%%%%%%%%%%%%%%%%%%%%%%%%%%%%%%%%%%%%%%%%%%%%%%%%%%%%%%

\section{Local energy decay for gravity-capillary waves} \label{s:nonlinear}

In this section we  prove our main result in Theorem~\ref{ThmG1}, modulo three results (see Lemmas~\ref{l:fixed-t},~\ref{l:kappa-diff},~\ref{l:kappa-3}),  whose proofs are in turn relegated to the last section of the paper.

We begin with a computation similar to \cite{AIT}. 
We use the density-flux pairs $(I_2,S_2)$ and $(I_3,S_3)$ introduced in the previous section. Given $\sigma\in (0,1)$ to be chosen later on, we set
\[
\mI_m^\sigma(t) = \int_\R m(x)( \sigma I_2(x,t) + (1-\sigma) I_3(x,t))\, dx .
\]
Since $\partial_t I_j+\partial_x S_j=0$, by integrating by parts, 
we have 
\[
\partial_t \mI_m^\sigma(t) =  \int_\R m_x ( \sigma S_2(x,t) + (1-\sigma) S_3(x,t)) \, dx.
\]
Consequently, to prove Theorem~\ref{ThmG1}, it is sufficient 
to establish the following estimates:
\begin{description}
\item[(i)] Fixed time bounds,
\begin{equation}\label{I2-bound}
\left|   \int_\R m(x) I_2 \, dx  \right| \lesssim 
\| \eta\|_{ g^{-\frac14}H^{\frac14}_h  \cap  
\, \kappa^{-\frac14} H^{\frac34}_h}
\| \psi_x\|_{g^{\frac14}H^{-\frac14}_h + \kappa^{\frac14}H^{-\frac34}_h},
\end{equation}
\begin{equation}\label{I3-bound}
\left|   \int_\R m(x) I_3 \, dx \right| 
\lesssim \| \eta\|_{g^{-\frac14}H^{\frac14}_h \cap \kappa^{-\frac14}H^{\frac34}_h} 
\|  \psi_x\|_{g^{\frac14}H^{-\frac14}_h+\kappa^{\frac14}H^{-\frac34}_h}.
\end{equation}

\item[(ii)] Time integrated bound; the goal is to prove that there exist 
$\sigma \in (0,1)$  and $c < 1$ such that
\begin{equation} \label{S23-bound}
\int_{0}^T \int_\R m_x ( \sigma S_2(t) + (1-\sigma) S_3(t))\, dx dt \gtrsim 
\| (\eta,\psi)\|_{LE^\kappa_{0}}^2 - c \| (\eta,\psi)\|_{LE^\kappa}^2,
\end{equation}
where the local norms are as defined in~\eqref{defiLEkappa}.
\end{description}
We now successively discuss the two sets of bounds above.

\bigskip

\textbf{(i) The fixed time bounds \eqref{I2-bound}-\eqref{I3-bound}.} These are similar to \cite{AIT}, except that in this case, one needs to take into account the threshold frequency $\lambda_{0}=\sqrt{g/\kappa}$. 
Since $I_2=\eta \psi_x$, by a duality argument, 
to prove the bound in \eqref{I2-bound}
it suffices to show that 
\begin{equation}
    \label{algebra}
    \Vert m\eta  \Vert_{g^{\frac{1}{4}}H^\frac{1}{4}_h\cap \kappa^{-\frac{1}{4}}H^{\frac{3}{4}}_h } \lesssim \Vert \eta  \Vert_{g^{\frac{1}{4}}H^\frac{1}{4}_h\cap \kappa^{-\frac{1}{4}}H^{\frac{3}{4}}_h }.
\end{equation}
The $H_h^\frac14$ bound was already proved in \cite{AIT}, so it remains to establish the $H^\frac34_h$ bound. For this we use 
a paradifferential decomposition of the product $m\eta$:
\[
m\eta =T_{m}\eta +T_{\eta}m +\Pi (m,\eta),
\]
and estimate each term separately. Here $T_{m}\eta$ represents the multiplication  between the low frequencies of $m$ and the high frequencies of $\eta$, and $\Pi(m, \eta)$ is the \emph{high-high} interaction operator. 

We begin with the first term and note that since we are 
estimating the high frequencies in  $\eta$, we do not have to 
deal with the $\lambda_0$ frequency threshold.
\[
\begin{aligned}
\Vert T_{m}\eta\Vert^2_{\dot{H}^{\frac{3}{4}} }&\lesssim \sum_{\lambda} \lambda^{\frac{3}{2}} \Vert m_{<\lambda}\eta_{\lambda}\Vert_{L^2}^2\lesssim  \sum_{\lambda}\lambda^{\frac{3}{2}}\Vert m_{< \lambda}\Vert^2_{L^{\infty}} \Vert \eta_{\lambda}\Vert_{L^2}^2 \lesssim \Vert m\Vert^2_{L^{\infty}}\sum_{\lambda} \lambda ^{\frac{3}{2}}\Vert \eta_{\lambda}\Vert_{L^2}^2 \\ &\lesssim \Vert m\Vert ^2_{L^{\infty}}\Vert \eta \Vert_{\dot{H}^{\frac{3}{4}}}.
\end{aligned}
\]

For the second term, we only need to bound \eqref{algebra} only at frequencies $\lambda$ with  $ \lambda_0 <\lambda $, as the case $\lambda_0 \geq \lambda $ follows as in \cite{AIT}. For this we rewrite $T_{\eta}m$
\[
T_{\eta}m= \sum_{\lambda } \eta_{<\lambda} m_{\lambda},
\]
and then, due to the fact that we are not adding all frequencies (only the ones above $\lambda_0$), we get
\[
\Vert T_{\eta}m \Vert_{\dot{H}^{\frac{3}{4}}} \lesssim  \sum_{\lambda_0<\lambda}\Vert  \eta_{<\lambda}m_{\lambda} \Vert^2_{\dot{H}^{\frac{3}{4}}},
\]
and for each term we estimate using Plancherel and Bernstein's inequalities  
\[
\begin{aligned}
\sum_{ \lambda_0< \lambda } \Vert \eta_{<\lambda } m_{\lambda} \Vert^2_{\dot{H}^{\frac{3}{4}}} &\lesssim \sum_{\lambda_{0}<\lambda }   \lambda ^{\frac{3}{2}}\Vert  \eta_{<\lambda} m_{\lambda}\Vert_{L^2}^2\\ 
&\lesssim \sum_{\lambda_{0}<\lambda }   \lambda ^{\frac{3}{2}}\Vert  \eta_{<\lambda} \Vert^2_{L^4} \Vert m_{\lambda}\Vert_{L^4}^2 \\ 
&\lesssim \sum_{\lambda_{0}<\lambda }   \lambda ^{\frac{3}{2}}\Vert  \eta \Vert^2_{\dot{H}^{\frac{1}{4}}} \lambda^{\frac{3}{2}}\Vert m_{\lambda}\Vert_{L^1}^2 \\
&\lesssim \sum_{\lambda_{0}<\lambda } \lambda^{-1}  \Vert  \eta \Vert^2_{\dot{H}^{\frac{1}{4}}} \Vert m_{xx}\Vert_{L^1}^2.
\end{aligned}
\]
The summation over $\lambda$ is trivial. Finally, the bound for the final term  is obtained in a similar fashion. 

To obtain the second 
bound in \eqref{I3-bound}, we begin by transforming $I_3$. Firstly, by definition 
of $I_3$ and $q$, we have 
\[
\int_\R m I_3 \, dx =\iint_{\Omega(t)} m(\theta_{y} \phi_x -\theta_{x} \phi_y)\, dy dx = \iint_{\Omega(t)} m\left( \partial_y (\theta \phi_x) -\partial_{x}(\theta \phi_y)\right)\, dydx.
\]

Now we have
\[
\iint_{\Omega(t)} m \partial_y (\theta \phi_x)\, dydx=\int_\R m (\theta \phi_x)\arrowvert_{y=\eta}\, dx.
\]
On the other hand, integrating by parts in $x$, we get
\[
\iint_{\Omega(t)} m \partial_{x}(\theta \phi_y)\, dydx=
-\iint_{\Omega(t)} m_x \theta \phi_{y}\, dy dx-\int_\R \eta_x m (\theta \phi_y)\arrowvert_{y=\eta}\, dx.
\]
Consequently,
\[
\int_\R m I_3 \, dx=\int_\R m (\theta \phi_x+\eta_x\theta\phi_y)\arrowvert_{y=\eta}\, dx
+\iint_{\Omega(t)} m_x \theta \phi_{y}\, dy dx.
\]
Now, by definition of $\theta$ one has $\theta\arrowvert_{y=\eta}=\eta$. 
Since $\phi_y=-q_x$ and since $(\phi_x+\eta_x\phi_y)\arrowvert_{y=\eta}=\psi_x$, we end up with
\[
\int_\R m I_3 \, dx=\int_\R mI_2\, dx-\iint_{\Omega(t)}m_x \theta q_x \, dydx.
\]

It remains to estimate the second part. This is a more delicate bound, which requires
the use of holomorphic coordinates and is postponed for the last section of the paper. We state
the desired bound as follows:
\begin{lemma}\label{l:fixed-t}
The following fixed estimate holds:
\begin{equation} \label{fixed-t}
\left | \iint_{\Omega} m_x \theta q_{x}\, dy dx \right | \lesssim   
\| \eta\|_{g^{-\frac{1}{4}}H^{\frac14}_h \cap \kappa^{-\frac{1}{4}}H^{\frac34}_h} \| \psi_x\|_{ g^{\frac{1}{4}}H^{-\frac14}_h+ \kappa^{\frac{1}{4}}H^{-\frac34}_h}.
\end{equation}
\end{lemma}

\bigskip
\textbf{ (ii) The time integrated bound \eqref{S23-bound}.}
We take $\sigma <1/2$, but close to
$1/2$. Using the expressions in Lemmas ~\ref{l:I2},~\ref{l:I3} 
as well as the relations \eqref{phit} and \eqref{thetat}
we write the integral in \eqref{S23-bound} as a combination of two
leading order terms plus error terms
\[
\int_{0}^T \int_\R m_x ( \sigma S_2(t) + (1-\sigma) S_3(t))\, dx dt = LE_{\psi} + g LE_\eta 
+ \kappa LE_\eta^\kappa
+ Err_1 + g Err_2 + Err_3,
\]
where   
\[
LE_\psi := \frac12 \int_0^T \iint_{\Omega(t)}  m_x [ \sigma(\phi_x^2 - \phi_y^2) + (1-\sigma) |\nabla \phi|^2] \, dx dy dt
\]
\[
LE_\eta := \int_0^T \left(\frac{\sigma}{2}  \int_\R m_x  \eta^2 \, dx - (1-\sigma) \iint_{\Omega(t)} m_x \theta_y ( \theta-\HN(\eta))\, dx dy\right) dt  ,
\]
\[
 LE_\eta^\kappa = \int_0^T \!\! \left(\int_\R m_x \left(  1 - \frac{1}{\sqrt{1+\eta_x^2}}
 - \sigma \eta H(\eta) \right)  dx- (1-\sigma) 
\iint_{\Omega(t)}\!\! m_x \theta_y H_N(H(\eta)) dy dx \!\right)\! dt ,
\]
and finally 
\[
\begin{aligned}
&Err_1 :=  \sigma \int_0^T \int_\R  m_x \eta \mN(\eta) \psi \,  dx dt ,\\
&Err_2 := \frac{1-\sigma}{2} \int_0^T \iint_{\Omega(t)} m_x 
\theta_y \HN(|\nabla \phi|^2) \, dx dy dt ,\\
&Err_3 :=  \frac{1-\sigma}{2} 
\int_0^T \iint_{\Omega(t)} m_x \phi_y \HD(\nabla \theta \nabla \phi)\,  dx dy dt .
\end{aligned}
\]

The terms which do not involve the surface tension have already been estimated 
in \cite{AIT}. We recall the outcome here:

\begin{proposition}[\cite{AIT}]
The following estimates hold:

(i) Positivity estimates: 
\begin{equation}
LE_\psi+ LE_\eta  \gtrsim \| (\eta,\psi)\|_{LE^\kappa_{0}}^2 - c \| (\eta,\psi)\|_{LE^\kappa}^2.
\end{equation}

(ii) Error bounds:
\begin{equation}
|Err_1|  + |Err_2 | \lesssim \epsilon LE(\eta,\psi).
\end{equation}

(iii) Normal form correction in holomorphic coordinates:
\begin{equation}
| Err_3 | \lesssim \epsilon \left( LE(\eta,\psi) + \| \eta(0)\|_{H^{\frac14}_h} \| \psi_x(0)\|_{H^{-\frac14}_h}
+ \| \eta(T)\|_{H^{\frac14}_h} \| \psi_x(T)\|_{H^{-\frac14}_h}\right).
\end{equation}

\end{proposition}

We recall that the bound for $Err_3$ is more complex because, rather than estimating it directly,
in \cite{AIT} we use a normal form correction to deal with the bulk of $Err_3$, and estimate 
directly only the ensuing remainder terms. Fortunately the normal form correction only 
uses the $\eta$ equation, and thus does not involve at all the surface tension.

Thus, in what follows our remaining task is to estimate the contribution of $LE_\eta^\kappa$,
which we describe in the following 

\begin{proposition}\label{p:kappa}
The following estimate holds:
\begin{equation}
LE_\eta^\kappa \gtrsim \int_\R m_x \eta_x^2 dx - \epsilon   LE^\kappa.
\end{equation}
\end{proposition}

For the rest of the section we consider the main steps in the proof of the above proposition.
We consider the three terms separately. For the first one there is nothing to do.
For the remaining two we  recall that
\begin{equation*}
H(\eta)=\px\left( \frac{\eta_x}{\sqrt{1+\eta_x^2}}\right).
\end{equation*}
The second  term is easier. Integrating by parts we obtain
\[
 \sigma \int_{0}^T \int_\R m_x  \left( \frac{\eta_x^2}{ \sqrt{1+\eta_x^2}} + 1 - \frac{1}{\sqrt{1+\eta_x^2}}\right)  
+  m_{xx} \frac{\eta \eta_x}{ \sqrt{1+\eta_x^2}}\, dx dt .
\]
The first term gives the positive contribution
\[
%\geq
c \int_0^T \int_\R m_x \eta_x^2 dx dt, \qquad c \approx \frac32 \sigma,
\]
while the second is lower order and can be controlled by Cauchy-Schwarz using 
the gravity part of the local energy, provided  that $\kappa \ll g$. 
This  condition is invariant with respect to pure time scaling, but 
not with respect to space-time scaling. This implies that even if this condition 
is not satisfied, we still have local energy  decay but with a window size 
larger than $1$ (depending on the ratio $\kappa/g$), provided that $h^2 g \gg \kappa$.

The more difficult  term is the last one, involving $\HN(H(\eta))$, namely 
\[
I^\kappa = - \int_0^T \iint_{\Omega(t)}  m_x \theta_y H_N(H(\eta))\,  dy dx dt. 
\]
The difficulty here is that, 
even though $H(\eta)$ is an exact derivative as a function of $x$, this property 
is lost when taking its harmonic extension since the domain itself is not flat. 
Thus in any natural expansion of $H_{N}(H(\eta))$ (e.g. in holomorphic coordinates 
where this is easier to see) there are quadratic (and also higher order) 
terms where no cancellation occurs in the 
high $\times$ high $\to$ low terms, making it impossible 
to factor out one derivative.

One can think of $I^\kappa$ as consisting of a leading order quadratic 
part in $\eta$ plus higher order 
terms. We expect the higher order terms to be perturbative 
because of our smallness condition,
but not the quadratic term. Because of this, 
it will help to identify precisely the quadratic term.
On the top we have, neglecting the quadratic and higher order terms,
\[
H(\eta) \approx \eta_{xx} \approx \theta_{xx},
\]
so one might think of replacing $H(\eta)$ with $\theta_{xx}$ modulo cubic and higher order 
terms.
This is not entirely correct since $\theta_{xx}$ satisfies a Dirichlet boundary condition on the bottom,
and not the Neumann boundary condition 
which we need. Nevertheless, we will still make this 
substitution, and pay the price of switching the boundary conditions. Precisely, we write
\begin{equation}\label{H-decomp}
H_N(H(\eta))  = \theta_{xx} + ( H_{N}(\theta_{xx}) - \theta_{xx})  + H_N(H(\eta)- \theta_{xx})
\end{equation}
and estimate separately the contribution of each term.

The contribution $I^\kappa_1$ of the first term in \eqref{H-decomp} to $I^\kappa$ is
easily described using the relation $\theta_{xx} = -\theta_{yy}$,
\begin{equation}\label{Ik1}
I^\kappa_1 = 
\int_0^T \iint_{\Omega(t)} m_x \theta_y \theta_{yy}\, dy dx dt = 
 \int_0^T \int_\R  m_x {\theta_y^2}\arrowvert_{y= \eta(t,x)}\,  dx dt, 
\end{equation}
which has the right sign.

It remains to estimate the integrals 
\[
I^\kappa_2 = - \int_0^T \iint_{\Omega(t)} m_x \theta_y(H_{N}(\theta_{xx}) - \theta_{xx})\, dy dx dt,
\]
and
\[
I^\kappa_3 = - \int_0^T \iint_{\Omega(t)} m_x \theta_y   H_N(H(\eta)- \theta_{xx}) \,   dy dx dt. 
\]

For these two integrals we will prove the following lemmas:

\begin{lemma}\label{l:kappa-diff}
The integral $I^\kappa_2$ is estimated by
\begin{equation}
| I^\kappa_{2}| \lesssim h^{-2} \| \eta \|_{LE}^2.
\end{equation}
\end{lemma}

%respectively 

\begin{lemma}\label{l:kappa-3}
The expression $I_{3}^{\kappa}$ is estimated by
\begin{equation}
 | I^{\kappa}_{3} | \lesssim \epsilon \Big(  \| \eta_x \|_{LE_{0}}^2  +  \frac{g}{\kappa}  \| \eta \|_{LE}^2\Big).
\end{equation}
\end{lemma}

Given the relation \eqref{Ik1} and the last two lemmas, the desired
result in Proposition~\ref{p:kappa} follows.  The two lemmas above are
most readily proved by switching to holomorphic coordinates. In the
next two sections we recall how the transition to holomorphic
coordinates works, following \cite{HIT}, \cite{ITc}, \cite{H-GIT} and
\cite{AIT}. Finally, in the last section of the paper we prove
Lemmas~\ref{l:kappa-diff},~\ref{l:kappa-3}.

\section{Holomorphic coordinates}
\label{ss:laplace}
%%%%%%%%%%%%%%%%%%%%%%%%%%%%%%%%%%%%%%%%%%%%%%%%%%%
%%%%%%%%%%%%%%%%%%%%%%%%%%%%%%%%%%%%%%%%%%%%%%%%%%%%%%%%%%%%%%%%%%%%
%%%%%%%%%%%%%%%%%%%%%%%%%%%%%%%%%%%%%%%%%%%%%%%%%%%%%%%%%%%%%%%%%%%%
%%%%%%%%%%%%%%%%%%%%%%%%%%%%%%%%%%%%%%%%%%%%%%%%%%%%%%%%%%%%%%%%%%%%
%%%%%%%%%%%%%%%%%%%%%%%%%%%%%%%%%%%%%%%%%%%%%%%%%%%%%%%%%%%%%%%%%%%%
%%%%%%%%%%%%%%%%%%%%%%%%%%%%%%%%%%%%%%%%%%%%%%%%%%%%%%%%%%%%%%%%%%%%
%%%%%%%%%%%%%%%%%%%%%%%%%%%%%%%%%%%%%%%%%%%%%%%%%%%%%%%%%%%%%%%%%%%%

\subsection{Harmonic functions in the canonical domain}

We begin by discussing two classes 
of harmonic functions in the horizontal strip \(\Strip = \Rl\times(-h,0)\). 

Given a function $f=f(\alpha)$ defined on the top, 
consider its harmonic extension with 
homogeneous Neumann boundary condition on the bottom,
\begin{equation}\label{MixedBCProblem}
\begin{cases}
\Delta u = 0\qquad \text{in}\ \ \Strip \\[1ex]
u(\alpha,0) = f\\[1ex]
\partial_\beta u(\alpha,-h) = 0.
\end{cases}
\end{equation}
It can be written in the form
\begin{equation}\label{PN}
u(\alpha,\beta)= P_N(\beta, D)f(\alpha): = \frac{1}{{2\pi}}\int p_N(\xi,\beta)\hat
f(\xi)e^{i\alpha\xi}\, d\xi,
\end{equation}
where $p_N$ is a Fourier multiplier with symbol
\[
p_N(\xi,\beta) = \frac{\cosh((\beta+h)\xi)}{\cosh(h\xi)}.
\]

We will make use of the Dirichlet to Neumann map $\mathcal{D}_N$, defined by
$$
\mathcal{D_N}f=\partial_\beta u(\cdot, 0),
$$
as well as the Tilbert transform, defined by
\begin{equation}\label{def:Tilbert}
\Tilh f(\alpha) =
-\frac{1}{2h}\lim\limits_{\epsilon\downarrow0}\int_{|\alpha-\alpha'|>\epsilon}
\cosech\left(\frac\pi{2h}(\alpha-\alpha')\right)f(\alpha')\, d\alpha'.
\end{equation}
Then the Tilbert transform is the Fourier multiplier 
\[
\Tilh = -i\tanh(hD).
\]
Notice that it 
takes real-valued functions to real-valued functions. The 
inverse Tilbert transform is denoted by \(\Tilh^{-1}\); \emph{ a priori} 
this is defined modulo constants. 
It follows that the Dirichlet to Neumann map can be written under the form
\[
\mathcal{D}_N f=\Tilh \partial_{\alpha} f.
\]

\bigskip
We now consider a similar problem 
with the homogeneous Dirichlet boundary condition on the bottom
\begin{equation}\label{DirichletProblem}
  \begin{cases}
    \Delta v = 0\qquad \text{in}\ \ \Strip \\[1ex]
    v(\alpha,0) = g \\[1ex]
     v(\alpha,-h) = 0.
  \end{cases}
\end{equation}
Then
\[
v(\alpha,\beta)=P_D(\beta, D)g(\alpha) := \frac{1}{{2\pi}}\int p_D(\xi,\beta)\hat
g(\xi)e^{i\alpha\xi}\, d\xi,
\]
where 
\[
p_D(\xi,\beta) = \frac{\sinh((\beta+h)\xi)}{\sinh(h\xi)}.
\]
The Dirichlet to Neumann map $\mathcal{D}_D$ for this problem is given by
\[
\partial_\beta v(\alpha ,0)=\mathcal{D}_D g=-\Til_h^{-1}\partial_{\alpha} g.
\]

\bigskip

The solution to~\eqref{MixedBCProblem} is related 
to the one of~\eqref{DirichletProblem} by means of harmonic conjugates.
Namely, given a real-valued solution $u$ to \eqref{MixedBCProblem},
we consider its 
harmonic conjugate $v$, i.e., 
satisfying the Cauchy-Riemann equations
\[
\left\{
\begin{aligned}
&u_\alpha = -v_\beta \\
&u_\beta = v_\alpha \\
&\partial_\beta u(\alpha,-h)=0.
\end{aligned}
\right.
\]
Then $v$ is a solution to \eqref{DirichletProblem} provided that 
the Dirichlet data $g$ for $v$ on the top is determined by the 
Dirichlet data $f$ for $u$ on the top via the relation
\[
g = -\Til_h f.
\]
Conversely, given $v$, there is 
a corresponding harmonic conjugate~$u$ (which is uniquely 
determined modulo real constants).

\subsection{Holomorphic functions in the canonical domain}

Here we consider the real algebra of 
holomorphic functions $w$ in the canonical domain 
$S\defn\{\alpha+i\beta\ :\ \alpha\in\xR,~-h\leq \beta\leq 0\}$, 
which are real on the bottom $\{\R-ih\}$. Notice that such 
functions are uniquely determined by their values on the top $\{\beta=0\}$,
and can be expressed 
as 
\[
w = u+ iv,
\]
where $u$ and $v$ are harmonic conjugate functions satisfying 
the equations \eqref{MixedBCProblem}, respectively \eqref{DirichletProblem}.

Hereafter, by definition, we will call 
functions on the real line holomorphic if they 
are the restriction on the real line of holomorphic functions in the 
canonical domain $S$ which are real on the bottom $\{\R-ih\}$. 
Put another way, they are functions $w\colon \xR\rightarrow \xC$ so 
that there is an holomorphic function, 
still denoted by $w\colon S\rightarrow \xC$, which satisfies
\[
\Im w = - \Tilh \Re w
\]
on the top.
The complex conjugates of holomorphic functions are called antiholomorphic. 

\subsection{Holomorphic coordinates and water waves}
Recall that $\Omega(t)$ denotes the fluid domain 
at a given time $t\ge 0$, in Eulerian coordinates. 
In this section we 
recall following \cite{HIT,ITc} (see also \cite{dy-zak,H-GIT}) 
how to rewrite the water-wave problem 
in holomorphic coordinates. 

We introduce holomorphic coordinates
$z=\alpha + i\beta$, thanks to conformal maps
\[
Z\colon S \rightarrow \Omega(t),
\]
which associate the top to the top, and the bottom to the bottom. 
Such a conformal transformation exists by the 
Riemann mapping theorem. Notice that these maps 
are uniquely defined up to horizontal translations in $S$ and that, 
restricted to the real axis, this provides a
parametrization for the water surface $\Gamma$. 
Because of the
boundary condition on the bottom of the fluid domain 
the function $W$ is holomorphic when $\alpha \in \mathbb{R}$. 

We set
$$
W:=Z-\alpha,
$$
so that $W = 0$ if the fluid surface is flat i.e., $\eta = 0$.

Moving to the velocity 
potential $\phi$, we consider its
harmonic conjugate~$q$ 
and then the function $Q:=\phi +i q$, 
taken in holomorphic coordinates, is the
holomorphic counterpart of $\phi$. 
Here $q$ is exactly the stream function, see \cite{AIT}.

With this notations, the water-wave problem can be recast as 
an evolution system for $(W, Q)$, 
within the space of holomorphic functions 
defined on the surface (again, we refer the reader 
to \cite{HIT,ITc,dy-zak,H-GIT} for the details of the computations). 
Here we recall the equations:
\begin{equation}\label{FullSystem-re}
\hspace{-.1in}
  \begin{cases}
    W_t + F(1+W_\alpha) = 0\vspace{0.1cm}\\
    Q_t + FQ_\alpha - g\Tilh[W] \! +\! \P_h \!
    \left[\dfrac{|Q_\alpha|^2}{J}\right] \!
    + \! i\kappa \P_h\! \left[ \dfrac{W_{\alpha \alpha}}{J^{1/2}(1+W_{\alpha})} -  \dfrac{\bar{W}_{\alpha \alpha}}{J^{1/2}(1+\bar{W}_{\alpha})} \right] \! = 0,
  \end{cases}
\end{equation}
where
\[
J = |1+W_\alpha|^2, \qquad F = \P_h\left[\frac{Q_\alpha-\bar
    Q_\alpha}{J}\right].
\]
Here $\P_h$ represents the orthogonal projection on the space of
holomorphic functions with respect with the inner product
in the Hilbert space $\mfH_h$ introduced in
\cite{H-GIT}. This has the form
\[
\langle u, v\rangle_{\mfH_h} := \int \left( \Til_h \Re u \cdot \Til_h \Re v  + \Im u \cdot \Im v \right) \, d\alpha, 
\]
and coincides with the $L^2$ inner product in the
infinite depth case.
Written in terms of the real and imaginary parts of $u$, the
projection $\P_h$ takes the form
\begin{equation}\label{def:P}
\P_h u = \frac12 \left[(1 - i \Tilh) \Re u + i (1+ i\Tilh^{-1}) \Im u           \right].
\end{equation}

Since all the functions in the system \eqref{FullSystem-re} are 
holomorphic, it follows that these relations also hold in the 
full strip $S$ for the holomorphic extensions of each term. 

We also remark that in the finite depth case there is an additional 
gauge freedom in the above form of the equations, in that $\Re F$ is a-priori 
only uniquely determined up to constants. This corresponds 
to the similar degree of freedom in the choice of the conformal coordinates,
and will be discussed in the last subsection.

A very useful function in the holomorphic setting is 
\[
R = \frac{Q_\alpha}{1+W_\alpha},
\]
which represents the ``good variable" in this setting, and  
corresponds to the Eulerian function
\[
R = \phi_x + i \phi_y.
\]
We also remark that the function $\theta$ introduced in the previous section is described in holomorphic coordinates by
\[
\theta =\Im W.
\]
Also related to $W$, we will use the auxiliary holomorphic function
\[
Y = \frac{W_\alpha}{1+W_\alpha}.
\]

Another important auxiliary function here is the advection velocity
 \[
 b = \Re F,
 \]
which represents the velocity of the particles on the fluid surface
in the holomorphic setting. 

It is also interesting to provide the form of the conservation 
laws in holomorphic coordinates.
We begin with the energy (Hamiltonian), which has the form
\[
\mathcal H =  \frac g2 \int |\Im W|^2 (1+\Re W_\alpha ) \, d\alpha
- \frac14\langle Q,\Tilh^{-1}[Q_\alpha]\rangle_{\mfH_h} .
\]
The momentum on the other has the form
\[
\mathcal M = \frac12 \langle W, \Til_h^{-1} Q_\alpha \rangle_{\mfH_h} 
= \int _{\mathbb{R}}\Til_h \Re W \cdot \Re Q_\alpha \, d\alpha = \int _{\mathbb{R}} \Im W \cdot \Re Q_\alpha \,  d\alpha .
\]

\subsection{Uniform bounds for the conformal map}
In order to freely switch computations between the Eulerian and holomorphic setting 
it is very useful to verify that our Eulerian  uniform smallness assumption 
for the functions $(\eta,\nabla \phi\arrowvert_{y=\eta})$ also has an identical 
interpretation in the holomorphic setting for the 
functions $(\Im W,R)$. Our main result is as follows:

\begin{theorem}\label{t:equiv}
Assume that the smallness condition \eqref{uniform} holds. Then we have 
\begin{equation}\label{X-kappa-hol}
\| (\Im W,R)\|_{X^\kappa} \lesssim \epsilon.
\end{equation}
\end{theorem}

This result is in effect an equivalence between the two bounds. We state and prove 
only this half because that is all that is needed here.

\begin{proof}
The similar result for the $X$ space corresponding to pure gravity waves was proved in \cite{AIT},
so we only need to add the $X_1$ component of the $X^\kappa$ norm. 
We first recall some of the set-up in \cite{AIT}, and then return to $X_1$.

The $X$ norm is described in \cite{AIT} using the language 
of frequency envelopes. We define a \emph{frequency envelope} 
for $(\eta,\nabla \phi_{|y = \eta})$ in $X$ to be any positive sequence 
\[
\left\{ c_{\lambda}\, : \, \quad h^{-1} < \lambda \in 2^\xZ \right\}
\]
with the following two properties:
\begin{enumerate}
\item Dyadic bound from above,
\[
\| P_\lambda (\eta,\nabla \phi_{|y = \eta})\|_{X_0} \leq c_\lambda .
\]
\item Slowly varying,
\[
 \frac{c_\lambda}{c_\mu} \leq \max 
 \left\{ \left(\frac{\lambda}{\mu}\right)^\delta, \left(\frac{\mu}{\lambda}\right)^\delta \right\} .
\]
\end{enumerate}
Here $\delta \ll 1$ is a small universal constant. 
Among all such frequency envelopes there exists a \emph{minimal frequency envelope}. 
In particular, this envelope has the property that
\[
\| (\eta,\nabla \phi_{|y = \eta})\|_{X} \approx \| c\|_{\ell^1}.
\]
We set the notations as follows:

\begin{definition} \label{def-fe}
By $\{c_\lambda\}_{\lambda \geq 1/h}$ we denote the minimal frequency 
envelope for $(\eta,\nabla \phi\arrowvert_{y=\eta})$ in $X_0$. 
We call $\{c_\lambda\}$ the \emph{control frequency envelope}.
\end{definition}
Since in solving the Laplace equation on the strip, solutions 
at depth $\beta$ are localized at frequencies $\leq \lambda$ 
where  $\lambda \approx  |\beta|^{-1}$, we will also use the notation 
\[
c_\beta = c_\lambda, \qquad \lambda \approx 
|\beta|^{-1}.
\]
This determines $c_\beta$ up to an $1+O(\epsilon)$ constant, 
which suffices for our purposes.

Using these notations, in \cite{AIT} we were able to prove 
a stronger version of the above theorem
for the $X$ norm,  and show that one can transfer the control envelope  
for $(\eta,\nabla \phi_{|y=\eta})$ to their counterpart $(\Im W,R)$
in the holomorphic coordinates.

\begin{proposition}\label{p:control-equiv}
Assume the smallness condition \eqref{uniform}, and let $\{c_\lambda\}$
be the control envelope as above. Then we have 
\begin{equation}\label{X-fe-hol}
\| P_{\lambda} (\Im W,R)\|_{X_0} \lesssim c_\lambda .
\end{equation}
\end{proposition}
As noted in \cite{AIT}, as a consequence of this proposition we can further 
extend the range of the frequency envelope estimates:

\begin{remark}
 The $X$ control envelope $\left\lbrace c_\lambda \right\rbrace$ 
 is also a frequency envelope for 

\begin{itemize}
\item $(\Im W, R)$ in $X_0$.

\item $W_\alpha$ in $H^{\frac12}_h$ and $L^\infty$.

\item $Y$ in $H^{\frac12}_h$.
\end{itemize}
\end{remark}

We remark that this in particular implies, by Bernstein's inequality, 
the pointwise bound
\begin{equation}\label{W-ppoint}
\| W_\alpha\|_{L^\infty} \lesssim \epsilon_0.
\end{equation}
This in turn implies that the Jacobian matrix for the change 
of coordinates stays close to the identity.

\smallskip

We now turn our attention to the $X_1$ component of the $X^\kappa$ norm. We 
have the additional information that 
\begin{equation}
\| \eta_{xx}\|_{L^2} \lesssim     \epsilon \left(\frac{g}{\kappa}\right)^\frac14 ,
\end{equation}
and we need to show that 
\begin{equation}\label{want-X1}
\| W_{\alpha\alpha}\|_{L^2} \lesssim    \epsilon \left(\frac{g}{\kappa}\right)^\frac14 . 
\end{equation}
We begin by computing
\[
\eta_{xx} = J^{-\frac12} \partial_\alpha( J^{-\frac12} \Im W_\alpha),
\]
therefore
\[
\| \eta_{xx} \|_{L^2} \approx \| \partial_\alpha( J^{-\frac12} \Im W_\alpha) \|_{L^2} .
\]
Thus we have 
\[
\| \Im W_{\alpha\alpha}\|_{L^2}  \lesssim \|\eta_{xx}\|_{L^2} 
+ \| W_\alpha \|_{L^\infty} \| W_{\alpha\alpha}\|_{L^2}
\lesssim \epsilon \left(\frac{g}{\kappa}\right)^\frac14 
+ \epsilon \| W_{\alpha\alpha}\|_{L^2}.
\]
On the other hand, the real and imaginary parts of $W_{\alpha\alpha}$ 
have the same regularity at frequency
$> h^{-1}$; more precisely, we can estimate
\[
\| \Re W_{\alpha\alpha}\|_{L^2}  \lesssim \|\Im W_{\alpha\alpha}\|_{L^2} + h^{-1} \| \Im W_\alpha\|_{L^2}
\lesssim 
\|\Im W_{\alpha\alpha}\|_{L^2} + h^{-\frac12} \epsilon,
\]
where the $L^2$ bound for $\Im W_\alpha$ comes from the $X$ norm.
Combining the last two bounds we get
\[
\| W_{\alpha\alpha}\|_{L^2}\lesssim \epsilon \left(\left(\frac{g}{\kappa}\right)^\frac14 + h^{-\frac12}\right).
\]
Then \eqref{want-X1} follows from the relation $h^2 g \gtrsim \kappa$, 
which says that the Bond number 
stays bounded.
\end{proof}

\subsection{Fixed time bounds at the level of the momentum}

Our objective here is to relate the Eulerian norms of $(\eta,\psi)$ 
at the momentum level in $E^\frac14$ (see \eqref{e14}) 
to their counterpart in the holomorphic setting
for $(W,R)$. Precisely, we have:

\begin{lemma} \label{l:equiv-M}
Assume that the condition \eqref{uniform} holds. 
Then we have the estimate
\begin{equation}
\| (\Im W,\partial_\alpha^{-1} \Im R) \|_{E^\frac14} \approx \| (\eta,\psi)\|_{E^\frac14}.     
\end{equation}
\end{lemma}

\begin{proof}
Recalling that $\eta = \Im W$, for the first part of 
the equivalence we are bounding the same function but 
in different coordinates. As the change of coordinates is bi-Lipschitz, 
the $L^2$ and $\dot H^1$ norms are equivalent, and, by interpolation, all
intermediate norms.

For the second part of the equivalence we use the relation $\psi =\Re Q$.
By the same reasoning as above, we can switch coordinates to get
\[
\| \psi \|_{ g^\frac14   \dot H^{\frac34}_h(\xR) +  \kappa^\frac14   \dot H^{\frac14}_h(\xR)}
\approx \| \Re Q \|_{ g^\frac14   \dot H^{\frac34}_h(\xR) +  h^\frac14   \dot H^{\frac14}_h(\xR)},
\]
where the first norm is relative to the  the Eulerian coordinate $x$
and the second norm is relative to the  the holomorphic coordinate $\alpha$.
It remains to relate the latter to the corresponding norm of $\partial^{-1} R$. 
Differentiating, we need to show that 
\[
\| Q_\alpha \|_{ g^{\frac14}  H^{-\frac14}_h(\xR) +  \kappa^\frac14 H^{-\frac34}_h(\xR)}
\approx \| R \|_{ g^{\frac14}  H^{-\frac14}_h(\xR) +  \kappa^\frac14 H^{-\frac34}_h(\xR)}.
\]
But here we can use the relation 
\[
Q_\alpha = R(1+W_\alpha),
\]
along with the multiplicative bound
\begin{equation}\label{algebra1}
\| f R \|_{g^\frac14  H^{-\frac14}_h(\xR) +  \kappa^\frac14 H^{-\frac34}_h(\xR)}
\lesssim (\| f\|_{L^\infty} + (\kappa/g)^\frac14 \| f_\alpha \|_{L^2}) 
\| R \|_{g^\frac14  H^{-\frac14}_h(\xR) +  \kappa^\frac14 H^{-\frac34}_h(\xR)},
\end{equation}
applied with $f = W_\alpha$ and then in the other direction with $f = Y$.

It remains to prove \eqref{algebra1}. By duality we rephrase this as 
\begin{equation}\label{algebra1-dual}
\| f R \|_{g^{-\frac14}  H^{\frac14}_h(\xR) \cap \kappa^{-\frac14} H^{\frac34}_h(\xR)}
\lesssim (\| f\|_{L^\infty}  + (\kappa/g)^\frac14 \| f_\alpha \|_{L^2}) 
\| R \|_{g^{-\frac14}H^{\frac14}_h(\xR) \cap \kappa^{-\frac14} H^{\frac34}_h(\xR)},
\end{equation}
which we approach in the same way as in the earlier proof of \eqref{algebra}.
In the paraproduct decomposition the terms $T_f R$ and $\Pi(f,R)$ are easy to estimate, 
using only the $L^\infty$ bound for $f$. The term $T_R f$ is more interesting.
At fixed frequency $\lambda$ we estimate in $L^2$ the product
\[
f_\lambda R_{< \lambda}.
\]
We split into two cases:

a) $\lambda < \lambda_0 = (g/\kappa)^\frac12$. Here we write
\[
\| f_\lambda R_{< \lambda} \|_{L^2} \lesssim \| f_\lambda \|_{L^4} 
\| R_{<\lambda}\|_{L^4} \lesssim (\kappa/g)^{-\frac18} (\| f\|_{L^\infty}  
+ (\kappa/g)^\frac14 \| f_\alpha \|_{L^2}) 
\| R \|_{g^{-\frac14}  H^{\frac14}_h(\xR)},
\]
which suffices.

b) $\lambda > \lambda_0 = (g/\kappa)^\frac12$. Then we estimate
\[
\| f_\lambda R_{< \lambda} \|_{L^2} \lesssim \| f_\lambda \|_{L^2} 
\| R_{<\lambda}\|_{L^\infty} \lesssim \lambda^{-1} (\kappa/g)^{-\frac18} [ (\kappa/g)^\frac14 
\| f_\alpha \|_{L^2}]  \| R \|_{g^{-\frac14}  H^{\frac14}_h(\xR)\cap \kappa^{-\frac14} H^{\frac34}_h(\xR)},
\]
which again suffices.

\end{proof}

\subsection{ Vertical strips in Eulerian vs holomorphic coordinates.} 

In our main result, the local energy functionals are defined using vertical strips
in Eulerian coordinates. On the other hand, for the multilinear
analysis in our error estimates in the last section, it would be easier to use vertical 
strips in holomorphic coordinates. To switch from one to 
the other we need to estimate the horizontal drift between the two strips in depth.  
As the conformal map is biLipschitz, it suffices to compare the 
centers of the two strips.  It is more convenient to do this in the reverse order, 
and compare the Eulerian image of the holomorphic vertical section with 
the Eulerian vertical section. This analysis was carried out in \cite{AIT}, 
and we recall the result here:

\begin{proposition} \label{p:switch-strips}
Let $(x_0, \eta(x_0))  = Z(\alpha_0,0)$, respectively $(\alpha_0,0)$ 
be the coordinates of a point on the free surface in Eulerian, 
respectively  holomorphic coordinates. 
Assume that \eqref{uniform} holds, and let $\{c_\lambda\}$ 
be the control frequency envelope in  Definition~\ref{def-fe}. 
Then we have the uniform bounds:
\begin{equation}
|\Re Z( \alpha_0, \beta) - x_0 + \beta \Im W_{\alpha}(\alpha_0,\beta)| 
\lesssim  c_\lambda, \qquad |\beta| \approx \lambda^{-1}.
\end{equation}
\end{proposition}

As a corollary, we see that the distance between the two 
strip centers grows at most linearly:

\begin{corollary}
Under the same assumptions as in the above proposition we have
\begin{equation}
|\Re Z( \alpha_0, \beta) - x_0| \lesssim \epsilon_0 |\beta|.
\end{equation}
\end{corollary}

\subsection{ The horizontal gauge invariance}

Here we briefly discuss the gauge freedom due to the fact that $\Re F$
is a-priori only uniquely determined up to constants.  In the infinite
depth case this gauge freedom is removed by making the assumption $F
\in L^2$. In the finite depth case (see \cite{H-GIT}) instead this is
more arbitrarily removed by setting $F(\alpha = -\infty) = 0$.

In the present paper no choice is necessary for our main result,
as well as for most of the proof. However, such a choice was made 
for convenience in \cite{AIT}, whose results we also apply here.
Thus we briefly recall it.

Assume first that we have a finite depth.
We start with a point $x_0 \in \R$ where our local energy estimate 
is centered.   Then we resolve the gauge invariance 
with respect to horizontal translations by setting $\alpha(x_0) = x_0$,
which corresponds to setting $ \Re W(x_0)= 0$. 
In dynamical terms, this implies that  the real part of $F$ is 
uniquely determined by 
\[
0 = \Re W_t(x_0) = \Re ( F(1+W_\alpha))(x_0),
\]
which yields
\[
\Re F(x_0) = \Im F(x,0) \frac{\Im W_\alpha(x_0)}{1+ \Re W_\alpha(x_0)}.
\]

In the infinite depth case, the canonical choice for $F$ is the one
vanishing at infinity. This corresponds to a moving location in the
$\alpha$ variable. We can still rectify this following the finite
depth model, at the expense of introducing a constant component in
both $\Re W$ and in $F$. We will follow this convention in the paper,
in order to insure that our infinite depth computation is an exact
limit of the finite depth case.

%%%%%%%%%%%%%%%%%%%%%%%%%%%%%%%%%%%%%%%%%%%%%%%%%%%%%%%%%%%%%%%%%%%%%%%%%%
%%%%%%%%%%%%%%%%%%%%%%%%%%%%%%%%%%%%%%%%%%%%%%%%%%%%%%%%%%%%%%%%%%%%%%%%%%
%%%%%%%%%%%%%%%%%%%%%%%%%%%%%%%%%%%%%%%%%%%%%%%%%%%%%%%%%%%%%%%%%%%%%%%%%%
%%%%%%%%%%%%%%%%%%%%%%%%%%%%%%%%%%%%%%%%%%%%%%%%%%%%%%%%%%%%%%%%%%%%%%%%%5
%%%%%%%%%%%%%%%%%%%%%%%%%%%%%%%%%%%%%%%%%%%%%%%%%%%%%%%%%%%%%%%%%%%%%%%%%%
%%%%%%%%%%%%%%%%%%%%%%%%%%%%%%%%%%%%%%%%%%%%%%%%%%%%%%%%%%%%%%%%%%%%%%%%%%
%%%%%%%%%%%%%%%%%%%%%%%%%%%%%%%%%%%%%%%%%%%%%%%%%%%%%%%%%%%%%%%%%%%%%%%%%%
%%%%%%%%%%%%%%%%%%%%%%%%%%%%%%%%%%%%%%%%%%%%%%%%%%%%%%%%%%%%%%%%%%%%%%%%%5
\section{Local energy bounds in holomorphic coordinates}
\label{s:switch}
%%%%%%%%%%%%%%%%%%%%%%%%%%%%%%%%%%%%%%%%%%%%%%%%%%%%%%%%%%%%%%%%%%%%%%%%%%
%%%%%%%%%%%%%%%%%%%%%%%%%%%%%%%%%%%%%%%%%%%%%%%%%%%%%%%%%%%%%%%%%%%%%%%%%%
%%%%%%%%%%%%%%%%%%%%%%%%%%%%%%%%%%%%%%%%%%%%%%%%%%%%%%%%%%%%%%%%%%%%%%%%%%
%%%%%%%%%%%%%%%%%%%%%%%%%%%%%%%%%%%%%%%%%%%%%%%%%%%%%%%%%%%%%%%%%%%%%%%%%5
%%%%%%%%%%%%%%%%%%%%%%%%%%%%%%%%%%%%%%%%%%%%%%%%%%%%%%%%%%%%%%%%%%%%%%%%%%
%%%%%%%%%%%%%%%%%%%%%%%%%%%%%%%%%%%%%%%%%%%%%%%%%%%%%%%%%%%%%%%%%%%%%%%%%%
%%%%%%%%%%%%%%%%%%%%%%%%%%%%%%%%%%%%%%%%%%%%%%%%%%%%%%%%%%%%%%%%%%%%%%%%%%
%%%%%%%%%%%%%%%%%%%%%%%%%%%%%%%%%%%%%%%%%%%%%%%%%%%%%%%%%%%%%%%%%%%%%%%%%5

\subsection{Notations}
We begin by transferring the 
local energy bounds to the holomorphic setting. Recall that 
in the Eulerian setting, they are equivalently defined as 
\[
\|(\eta,\psi)\|_{LE^\kappa} := g^\frac12 \|\eta\|_{LE^0} + \kappa^\frac12 \|\eta_x\|_{LE^0} 
+ \| \nabla \phi\|_{LE^{-\frac12}},
\]
where 
\[
\|\eta\|_{LE^0} := \sup_{x_0 \in \R} \| \eta\|_{L^2(\HS(x_0))},
\qquad
\| \nabla \phi\|_{LE^{-\frac12}} := \sup_{x_0 \in \R} \| \nabla \phi\|_{L^2(\HVS(x_0))}.
\]
Here $\HS(x_0)$, respectively $\HVS(x_0)$ represent the Eulerian strips
\[
\HS(x_0) := \{ [0,T] \times [x_0-1,x_0+1] \}, \qquad
\HVS(x_0) := \HS(x_0)\times [-h,0].
\]

%A first objective here is to prove that these norms 
%are equivalent to their 
%counterparts in the holomorphic setting. 
In holomorphic coordinates the 
functions $\eta$ and $\nabla \phi$ are given by $\Im W$ and $R$. 
Thus we seek to replace the above local energy norm with
\[
\|(W,R)\|_{LE} := \|\Im W\|_{LE^0} + \| R \|_{LE^{-\frac12}},
\]
with
\[
\|\Im W\|_{LE^0} := \sup_{x_0 \in \R} \| \Im W \|_{L^2(\HS_h(x_0))},
\qquad
\| R \|_{LE^{-\frac12}} := 
\sup_{x_0 \in \R} \| R \|_{L^2(\HVS_h(x_0))}.
\]
Here $\HS_h(x_0)$ and $\HVS_h(x_0)$ represent the holomorphic strips given by
\[
\HS_h(x_0) := \{(t,\alpha): t \in [0,T], \ \alpha \in [\alpha_0-1,\alpha_0+1]\}, \qquad
\HVS_h(x_0) := \HS(x_0) \times [-h,0],
\]
where $\alpha_0 = \alpha_0(t,x_0)$ represents the holomorphic 
coordinate of $x_0$, which in general will depend on $t$. 

We call the attention to the fact that, 
while the strips $\HS_h(x_0)$ on the top roughly correspond 
to the image of $\HS(x_0)$ in holomorphic coordinates, 
this is not the case for the strips $\HVS_h(x_0)$ 
relative to $\HVS(x_0)$. In depth, there may be a horizontal drift, 
which is estimated by means of Proposition~\ref{p:switch-strips}.

We can now state the following equivalence.
\begin{proposition}\label{p:equiv}
Assuming the uniform bound \eqref{uniform}, we have the equivalence:
\begin{equation}\label{equiv}
\|(\eta,\psi)\|_{LE^\kappa} \approx \|(W,R)\|_{LE^\kappa}.
\end{equation}
\end{proposition}
\begin{proof}
Here the correspondence between the $LE^0$ norms of $\eta$ and $\Im W$ 
is straightforward due to the bi-Lipschitz property of the conformal map. 
However, the correspondence between the $LE^{-\frac12}$ norms of $\nabla \phi$ and $R$ 
is less obvious, and was studied in detail in \cite{AIT}.

Moving on to the $LE^0$ norms of $\eta_x$ and $\Im W$, we have 
\[
\eta_x = J^{-\frac12} \Im W_\alpha.
\]
Since $J= 1+O(\epsilon)$ and the correspondence between 
the two sets of coordinates is bi-Lipschitz, it immediately follows that 
$\| \eta_x \|_{LE^0} \approx \| W_\alpha \|_{LE^0}$.
\end{proof}

One difference between the norms for $\Im W$ and for $R$ is that they
are expressed in terms of the size of the function on the top,
respectively in depth. For the purpose of multilinear estimates later
on we will need access to both types of norms.
Since the local energy norms are defined using the unit spatial
scale, in order to describe the behavior of functions in these spaces
we will differentiate between high frequencies and low frequencies. We
begin with functions on the top:

\medskip

a) \textbf{ High frequency characterization on top.} Here we will use local norms
on the top, for which we will use the abbreviated notation
\[
\| u \|_{L^2_t H^s_{loc}} := \sup_{x_0 \in \R } \| u\|_{L^2_t H^s_{\alpha}([\alpha_0-1, \alpha_0+1])},
\]
where again $\alpha_0 = \alpha_0(x_0,t)$.

\medskip

b) \textbf{ Low frequency characterization on top.} Here we will use local norms
on the top to describe the frequency $\lambda$ 
or $\leq \lambda$ part of functions, where $\lambda < 1$ is a dyadic frequency. By the 
uncertainty principle such bounds should be uniform on the $\lambda^{-1}$ 
spatial scale. Then it is natural to use the following norms:
\[
\| u \|_{L^2_t L^\infty_{loc}(B_\lambda)} := \sup_{x_0 \in \R } \| u\|_{L^2_t L^\infty_{\alpha}(B_\lambda(x_0))},
\]
where
\[
B_\lambda(x_0) := \{ \alpha \in \R: \ |\alpha-\alpha_0| \lesssim \lambda^{-1} \}.
\]
We remark that the local norms in $a)$ correspond 
exactly to the $B_{\lambda}(x_0)$ norms with $\lambda =1$.

\bigskip
Next we consider functions in the strip which are harmonic extensions 
of functions on the top. To measure them we will use function spaces as follows:

\medskip

a1) \textbf{ High frequency characterization in strip.} 
Here we will use local norms on regions with depth at most $1$, 
for which we will use the abbreviated notation
\[
\| u \|_{L^2_t X_{loc}(A_1)} := \sup_{x_0 \in \R } \| u\|_{L^2_t X(A_1(x_0))},
\]
where $X$ will represent various Sobolev norms and
\[
A_1(x_0) := \{(\alpha,\beta): |\beta| \lesssim 1,\ |\alpha-\alpha_0| \lesssim 1 \}.
\]

\medskip

b1) \textbf{ Low frequency characterization in strip.} Here a 
frequency $\lambda < 1$ is associated 
with depths $|\beta| \approx \lambda^{-1}$. 
Thus, we define the regions
\[
A_{\lambda}(x_0) = \{ (\alpha,\beta): |\beta| \approx \lambda^{-1},\  |\alpha-\alpha_0|
\lesssim \lambda^{-1} \}, \qquad \lambda < 1,
\]
as well as
\[
\begin{aligned}
&\mathbf{B}_1(x_0) := \{ (\alpha,\beta); \ |\alpha - \alpha_0| \leq 1, \ \beta \in [-1,0] \} ,\\
&\mathbf{B}_\lambda (x_0)
:= \{ (\alpha,\beta); \ |\alpha - \alpha_0| \leq \lambda^{-1}, \ \beta \in [-\lambda^{-1},0] \}, 
\mbox{ for } \lambda < 1. 
\end{aligned}
\]
In these regions we use the uniform norms,
\[
\| u \|_{L^2_t L^\infty_{loc}(A_\lambda)} := \sup_{x_0 \in \R }
\| u\|_{L^2_t L^\infty_{\alpha,\beta}(A_\lambda(x_0))},
\]
and similarly for $\mathbf{B}^1$ and $\mathbf{B}_\lambda$. 

% To simplify the notations in the following analysis,
% we will also denote
% \begin{equation}\label{constants}
% \| (\eta,\psi)\|_{LE^\kappa} := M, \qquad \|(\eta,\psi)\|_{X} := \epsilon \leq \epsilon_0 \ll 1.
% \end{equation}

% Given the equivalence of the $X$ norms in Theorem~\ref{t:equiv},
% as well as the equivalence of the $LE$ norms in the next subsection,
% these bounds also transfer to the holomorphic setting as follows:
% \begin{equation}\label{constants-hol}
% \| (\Im W,R)\|_{LE^\kappa} \lesssim M, \qquad \|(\Im W,R)\|_{X} \lesssim \epsilon \ll 1.
% \end{equation}
% Furthermore, we recall that the frequency envelopes $\{ c_\lambda \}$ for $(\eta,\psi)$ in $X_0$ also transfer to  $(\Im W,R)$ in $X_0$.

%%%%%%%%%%%%%%%%%%%%%%%%%%%%%%%%%%%%%%%%%%%%%%%%%%%%%%%%%%%%%%%%%%%%%%%%%%
%%%%%%%%%%%%%%%%%%%%%%%%%%%%%%%%%%%%%%%%%%%%%%%%%%%%%%%%%%%%%%%%%%%%%%%%%%
%%%%%%%%%%%%%%%%%%%%%%%%%%%%%%%%%%%%%%%%%%%%%%%%%%%%%%%%%%%%%%%%%%%%%%%%%%
%%%%%%%%%%%%%%%%%%%%%%%%%%%%%%%%%%%%%%%%%%%%%%%%%%%%%%%%%%%%%%%%%%%%%%%%%5
\subsection{Multipliers and Bernstein's inequality in uniform norms}
%%%%%%%%%%%%%%%%%%%%%%%%%%%%%%%%%%%%%%%%%%%%%%%%%%%%%%%%%%%%%%%%%%%%%%%%%%
%%%%%%%%%%%%%%%%%%%%%%%%%%%%%%%%%%%%%%%%%%%%%%%%%%%%%%%%%%%%%%%%%%%%%%%%%%
%%%%%%%%%%%%%%%%%%%%%%%%%%%%%%%%%%%%%%%%%%%%%%%%%%%%%%%%%%%%%%%%%%%%%%%%%%
%%%%%%%%%%%%%%%%%%%%%%%%%%%%%%%%%%%%%%%%%%%%%%%%%%%%%%%%%%%%%%%%%%%%%%%%%5
Here we recall the results of \cite{AIT} describing how multipliers 
act on the uniform spaces defined above. We will work 
with a multiplier $M_{\lambda}(D)$ associated to a dyadic frequency $\lambda$. 
In order to be able to use the bounds in several circumstances, 
we make a weak assumption on their (Lipschitz) 
symbols $m_{\lambda}(\xi)$: 
\begin{equation}
\label{m-symbol}
\begin{aligned}
|m_{\lambda}(\xi)| \lesssim  \ (1+\lambda^{-1}|\xi|)^{-3}, \ 
\mbox{  and  }\ 
| \partial_\xi^{k+1} m_{\lambda}(\xi) | \lesssim c_k  \ |\xi|^{-k}(1+\lambda^{-1}|\xi|)^{-4}. 
\end{aligned}
\end{equation}
Examples of such symbols include
\begin{itemize}
\item Littlewood-Paley localization operators $P_{\lambda}$, $P_{\leq \lambda}$.
\item The multipliers $p_D(\beta, D)$ and $p_N(\beta, D)$ in subsection~\ref{ss:laplace}
with $|\beta| \approx \lambda^{-1}$.
\end{itemize}

We will separately consider high frequencies, where 
we work with the spaces  $L^2_t L^p_{loc}$, 
and low frequencies, where we work with the spaces  
$L^2_t L^p_{loc}(B_\lambda)$ associated with a dyadic 
frequency $1/h \leq \lambda \leq 1$.

\bigskip

\textbf{A. High frequencies.}
Here we consider a dyadic high frequency $\lambda \geq 1$,
and seek to understand how multipliers $M_{\lambda}(D)$
associated to frequency $\lambda$ act on the spaces $L^2_t L^p_{loc}$.

\begin{lemma}\label{l:bern-loc-hi}
Let $\lambda \geq 1$ and $1 \leq p \leq q \leq \infty$.  Then
\begin{equation}
\| M_{\lambda}(D) \|_{L^2_t L^p_{loc} \to 
L^2_t L^q_{loc}} \lesssim \lambda^{\frac1p-\frac1q}.
\end{equation}
\end{lemma}

\bigskip

\textbf{B. Low frequencies.}
Here we consider two dyadic low frequencies $1/h \leq \lambda_1, \lambda_2 \leq 1$,
and seek to understand how multipliers $M_{\lambda_2}(D)$
associated to frequency $\lambda_2$ 
act on the spaces $L^2_t L^p_{loc}(B_{\lambda_1})$. 
For such multipliers we have:

\begin{lemma}\label{l:bern-loc}
Let $1/h \leq \lambda_1,\lambda_2 \leq 1$ and $1 \leq p \leq q \leq \infty$.

a) Assume that $\lambda_1 \leq \lambda_2$. Then
\begin{equation}
\| M_{\lambda_2}(D) \|_{L^2_t L^p_{loc}(B_{\lambda_1}) \to 
L^2_t L^q_{loc}(B_{\lambda_1})} \lesssim \lambda_2^{\frac1p-\frac1q}.
\end{equation}

b) Assume that $ \lambda_2 \leq \lambda_1$. Then
\begin{equation}
\| M_{\lambda_2}(D) \|_{L^2_t L^p_{loc}(B_{\lambda_1}) \to 
L^2_t L^q_{loc}(B_{\lambda_2})} \lesssim \lambda_1^{\frac1p} \lambda_2^{-\frac1q}.
\end{equation}
\end{lemma}
We remark that part (a) is nothing but the classical Bernstein's
inequality in disguise, as the multiplier $M_{\lambda_2}
$ does not mix
$\lambda_1^{-1}$ intervals. Part (b) is the more interesting one,
where the $\lambda_1^{-1}$ intervals are mixed.

%%%%%%%%%%%%%%%%%%%%%%%%%%%%%%%%%%%%%%%%%%%%%%%%%%%%%%%%%%%%%%%%%%%%%%%%%%
%%%%%%%%%%%%%%%%%%%%%%%%%%%%%%%%%%%%%%%%%%%%%%%%%%%%%%%%%%%%%%%%%%%%%%%%%%
%%%%%%%%%%%%%%%%%%%%%%%%%%%%%%%%%%%%%%%%%%%%%%%%%%%%%%%%%%%%%%%%%%%%%%%%%%
%%%%%%%%%%%%%%%%%%%%%%%%%%%%%%%%%%%%%%%%%%%%%%%%%%%%%%%%%%%%%%%%%%%%%%%%%%
%%%%%%%%%%%%%%%%%%%%%%%%%%%%%%%%%%%%%%%%%%%%%%%%%%%%%%%%%%%%%%%%%%%%%%%%%%
%%%%%%%%%%%%%%%%%%%%%%%%%%%%%%%%%%%%%%%%%%%%%%%%%%%%%%%%%%%%%%%%%%%%%%%%%%
\subsection{Bounds for \texorpdfstring{$\eta = \Im W$}{} and for 
\texorpdfstring{$\eta_x = J^{-\frac12} \Im W_\alpha$.}{}}
%%%%%%%%%%%%%%%%%%%%%%%%%%%%%%%%%%%%%%%%%%%%%%%%%%%%%%%%%%%%%%%%%%%%%%%%%%
%%%%%%%%%%%%%%%%%%%%%%%%%%%%%%%%%%%%%%%%%%%%%%%%%%%%%%%%%%%%%%%%%%%%%%%%%%
%%%%%%%%%%%%%%%%%%%%%%%%%%%%%%%%%%%%%%%%%%%%%%%%%%%%%%%%%%%%%%%%%%%%%%%%%%
%%%%%%%%%%%%%%%%%%%%%%%%%%%%%%%%%%%%%%%%%%%%%%%%%%%%%%%%%%%%%%%%%%%%%%%%%5

Here we have the straightforward equivalence
\begin{equation}\label{W-eta-le}
\| \eta \|_{LE^0} \approx \| \Im W\|_{LE^0}, 
\qquad \| \eta_x \|_{LE^0} \approx \| \Im W_\alpha\|_{LE^0}, 
\end{equation}
as $\eta$ and $\Im W$ are one and the same function up to a bi-Lipschitz 
change of coordinates. We begin with 
a bound from \cite{AIT} for the low frequencies of $\Im W$ on the top:

\begin{lemma}\label{l:W-LE-top}
For each dyadic frequency $1/h \leq \lambda < 1$ we have
\begin{equation}\label{W-LE-top}
\|\Im W_{\leq \lambda} \|_{L^2_t L^\infty_{loc}(B_\lambda)} \lesssim \| \Im W\|_{LE^0}. 
\end{equation}
\end{lemma}
Here one may also replace $\Im W$ by $W_\alpha$,
\begin{equation}\label{Wa-LE-top}
\|W_{\alpha,\leq \lambda} \|_{L^2_t L^\infty_{loc}(B_\lambda)} \lesssim \|W_\alpha\|_{LE^0}. 
\end{equation}
Since 
\[
\|W_\alpha\|_{LE^0} \lesssim \| \Im W_\alpha\|_{LE^0} + h^{-1} \| \Im W\|_{LE^0},
\]
we can also estimate the same expression in terms of $\Im W$,
\begin{equation}\label{Wa-LE-top+}
\|W_{\alpha,\leq \lambda} \|_{L^2_t L^\infty_{loc}(B_\lambda)} \lesssim \lambda \|\Im W\|_{LE^0}. 
\end{equation}

On the other hand, for nonlinear estimates, we also need bounds in depth,
precisely over the regions $A_{\lambda}(x_0)$.  There, by \cite{AIT}, we have

\begin{lemma}\label{l:theta-LE-loc}
For each dyadic frequency $\lambda < 1$ we have
\begin{equation}\label{theta-LE-loc}
\|\Im W \|_{L^2_t L^\infty_{loc}(A_\lambda)} 
+\lambda^{-1}\|W_{\alpha} \|_{L^2_t L^\infty_{loc}(A_\lambda)} \lesssim \| \Im W\|_{LE^0}. 
\end{equation}
\end{lemma}

We will also need a mild high frequency bound:

\begin{lemma}\label{l:theta-LE-hi}
The following rstimate holds:
\begin{equation}\label{theta-LE-hi}
\|\Im W \|_{L^2_t L^2_{loc}(A_1)} + \|\beta W_{\alpha} \|_{L^2_t L^\infty_{loc}(A_1)} 
\lesssim \| \Im W\|_{LE^0}. 
\end{equation}
\end{lemma}

\begin{proof}
It suffices to prove a fixed $\beta$ bound,
\[
\|\Im W(\beta) \|_{L^2_t L^2_{loc}(B_1)} 
+ \|\beta W_{\alpha}(\beta) \|_{L^2_t L^\infty_{loc}(B_1)} \lesssim \| \Im W\|_{LE^0}, \qquad \beta \in (0,1). 
\]
But both $\Im W(\beta)$ and $\beta W_{\alpha}(\beta)$ are defined in terms 
of $\Im W$ via zero order multipliers localized at frequency $\leq \beta^{-1}$.
Hence the above estimate easily follows.
\end{proof}

%%%%%%%%%%%%%%%%%%%%%%%%%%%%%%%%%%%%%%%%%%%%%%%%%%%%%%%%%%%%%%%%%%%%%%%%%%
%%%%%%%%%%%%%%%%%%%%%%%%%%%%%%%%%%%%%%%%%%%%%%%%%%%%%%%%%%%%%%%%%%%%%%%%%%
%%%%%%%%%%%%%%%%%%%%%%%%%%%%%%%%%%%%%%%%%%%%%%%%%%%%%%%%%%%%%%%%%%%%%%%%%%
%%%%%%%%%%%%%%%%%%%%%%%%%%%%%%%%%%%%%%%%%%%%%%%%%%%%%%%%%%%%%%%%%%%%%%%%%5
\subsection{ Estimates for \texorpdfstring{$F(W_\alpha$)}{}}
%%%%%%%%%%%%%%%%%%%%%%%%%%%%%%%%%%%%%%%%%%%%%%%%%%%%%%%%%%%%%%%%%%%%%%%%%%
%%%%%%%%%%%%%%%%%%%%%%%%%%%%%%%%%%%%%%%%%%%%%%%%%%%%%%%%%%%%%%%%%%%%%%%%%%
%%%%%%%%%%%%%%%%%%%%%%%%%%%%%%%%%%%%%%%%%%%%%%%%%%%%%%%%%%%%%%%%%%%%%%%%%%
%%%%%%%%%%%%%%%%%%%%%%%%%%%%%%%%%%%%%%%%%%%%%%%%%%%%%%%%%%%%%%%%%%%%%%%%%5

Here we consider a function $F$ which is holomorphic in a 
neighbourhood of $0$, with $F(0) = 0$ and prove 
local energy bounds for the auxiliary  holomorphic function $F(W_\alpha)$.

\begin{lemma}[ Holomorphic Moser estimate] 
\label{l:Y}
Assume that $\|W_\alpha\|_{ L^\infty} \ll 1$. Then

a) For $\lambda > 1$ we have 
\begin{equation}\label{Yhi}
\| F(W_\alpha)_\lambda \|_{L^2_t L^2_{loc}} \lesssim \lambda \|W\|_{L_t^2 L^2_{loc}}.
\end{equation}

b) For $\lambda \leq 1$ we have 
\begin{equation}\label{Ylo}
\| F(W_\alpha)_\lambda \|_{L^2_t L_{\alpha}^\infty(B_\lambda(x_0))} \lesssim \lambda\|W\|_{L_t^2 L^2_{loc}}.
\end{equation}
\end{lemma}
In both cases, one should think of the implicit 
constant as depending on $\|W_\alpha\|_{L^\infty}$.
We note that both estimates follow directly from 
Lemma~\ref{l:theta-LE-loc} if $F(z) = z$. However
to switch to an arbitrary $F$ one would seem 
to need some Moser type inequalities, 
which unfortunately do not work in negative Sobolev spaces.
The key observation is that in both of these 
estimates it is critical that $W_\alpha$ is holomorphic, and $F(W_\alpha)$ is an analytic function 
of $W_\alpha$. This lemma was proved in \cite{AIT} for the expression
\[
Y = \frac{W_\alpha}{1+W_\alpha},
\]
but the proof is identical in the more general case considered here.

%%%%%%%%%%%%%%%%%%%%%%%%%%%%%%%%%%%%%%%%%%%%%%%%%%%%%%%%%%%%%%%%%%%%%%%%%%
%%%%%%%%%%%%%%%%%%%%%%%%%%%%%%%%%%%%%%%%%%%%%%%%%%%%%%%%%%%%%%%%%%%%%%%%%%
%%%%%%%%%%%%%%%%%%%%%%%%%%%%%%%%%%%%%%%%%%%%%%%%%%%%%%%%%%%%%%%%%%%%%%%%%%
%%%%%%%%%%%%%%%%%%%%%%%%%%%%%%%%%%%%%%%%%%%%%%%%%%%%%%%%%%%%%%%%%%%%%%%%%5

\section{ The error estimates}

The aim of this section is to use holomorphic coordinates in order to  prove
Lemmas~\ref{l:fixed-t},~\ref{l:kappa-diff},~\ref{l:kappa-3}, which for convenience we recall below.

\begin{lemma}\label{l:fixed-t-re}
The following fixed estimate holds:
\begin{equation} \label{fixed-t-re}
\left | \iint_{\Omega(t)} m_x(x-x_0) \theta q_{x}\, dy dx \right | \lesssim   
\| \eta\|_{g^{-\frac{1}{4}}H^{\frac14}_h \cap \kappa^{-\frac{1}{4}}H^{\frac34}_h} \| \psi_x\|_{ g^{\frac{1}{4}}H^{-\frac14}_h+ h^{\frac{1}{4}}H^{-\frac34}_h}.
\end{equation}
\end{lemma}

\begin{lemma}\label{l:kappa-diff-re}
The integral $I^\kappa_2$ is estimated by
\begin{equation}
| I^\kappa_{2}| \lesssim h^{-2} \| \eta \|_{LE^0}^2 .
\end{equation}
\end{lemma}

%respectively 

\begin{lemma}\label{l:kappa-3-re}
The expression $I_{3}^{\kappa}$ is estimated by
\begin{equation}
 | I^{\kappa}_{3} | \lesssim \epsilon \Big(  \| \eta_x \|_{LE^{0}}^2  +  \frac{g}{\kappa}  \| \eta \|_{LE^{\kappa}}^2 \Big).
\end{equation}
\end{lemma}

We begin by expressing the quantities in the Lemma using holomorphic coordinates.
We first recall that 
\[
\theta = \Im W, \qquad q_x = \Im R,
\]
and by the chain rule we have 
\[
 \theta_x = \Im \left(\frac{W_\alpha}{1+W_\alpha}\right), \qquad \theta_y = 
\Re \left(\frac{W_\alpha}{1+W_\alpha}\right).
\]
A second use of chain rule yields 
\begin{equation}
    \label{second}
 \theta_{xx} = \Im \left[ \frac{1}{1+W_\alpha} 
 \partial_\alpha  \left(\frac{W_\alpha}{1+W_\alpha}\right) \right] 
 = -\frac12 \partial_\alpha \Im \frac{1}{(1+W_\alpha)^2}.
\end{equation}
Finally,  $H(\eta)$ is expressed as 
\[
H (\eta) = - \frac{i}{1+ W_\alpha} \partial_\alpha (J^{-\frac12}(1+ W_\alpha)).
\]

We recall that the local energy norms are easily transferred,
see Proposition~\ref{p:equiv}
\[
\| \eta \|_{LE^{\kappa}} \approx \| \Im W \|_{LE^{\kappa}}, 
\qquad \| \eta_x \|_{LE^{\kappa}} \approx \| \Im W_\alpha \|_{LE^{\kappa}}.
\]
while, by Theorem~\ref{t:equiv}, for the uniform bound we have 
\begin{equation}\label{l:transfer-hf}
\| \Im W\|_{X^{\kappa}}  \lesssim \| \eta \|_{X^{\kappa}}.
\end{equation}
The last item we need to take into account in switching coordinates is that 
the image of the vertical strip in Euclidean coordinates is still 
a strip $S_{hol}(t)$  in the holomorphic setting with $O(1)$ horizontal size, 
but centered around  $\alpha_0(t,\beta)$, where 
\[
|\alpha_0(t,\beta) - \alpha_0(t,0) | \lesssim \epsilon |\beta|.
\] 
This is from Proposition~\ref{p:switch-strips}.

\subsection{ Proof of Lemma~\ref{l:fixed-t-re}}

We begin by rewriting our integral in holomorphic coordinates,
\[
I_{t,x_0} = \iint_S J m_x(x-x_0) \Im W \Im R \, d\alpha d\beta.
\]
In view of the norm equivalence in Lemma~\ref{l:equiv-M}, 
for this integral we need to prove the bound
\begin{equation} \label{fixed-t-hol}
| I_{t,x_0} | \lesssim   \| \Im W \|_{g^{-\frac{1}{4}}H^{\frac14}_h \cap \kappa^{-\frac{1}{4}}H^{\frac34}_h} \| \Im R \|_{ g^{\frac{1}{4}}H^{-\frac14}_h+ \kappa^{\frac{1}{4}}H^{-\frac34}_h} 
\end{equation}
where the norms on the right are taken on the top.
Here we recall that $m_x$ is a bounded, 
Lipschitz bump function with support 
in the strip $S_{hol}(t)$. This is all we will use concerning $m_x$.
The strip $S_{hol}(t)$ is contained in the dyadic union
\[
S_{hol}(t) \subset A_1(x_0) \bigcup \bigcup_{ h^{-1} < \lambda < 1} A_\lambda(x_0). 
\]
Correspondingly we split the integral as 
\[
I_{t,x_0} = I_1 + \sum_{ h^{-1} < \lambda < 1} I_\lambda.
\]
For $I_\lambda$ we directly estimate
\[
|I_\lambda| \lesssim  \lambda^{-1} \| \Im W\|_{L^\infty(A_\lambda(x_0))} 
\| \Im R\|_{L^\infty(A_\lambda(x_0))}.
\]
For the pointwise bounds we recall that 
$\Im W(\beta) = P_N(\beta,D) \Im W $, and similarly for $\Im R$,
where, for $\beta \approx \lambda^{-1}$, 
the multiplier $P_N(\beta,D)$ selects the frequencies 
$\leq \lambda$. Hence, harmlessly allowing 
rapidly decaying tails in our Littlewood-Paley truncations, 
we obtain using Bernstein's inequality
\[
\begin{split}
\sum_{\lambda < 1} |I_\lambda| \lesssim & \  \sum_{\lambda < 1} \lambda^{-1} 
\| \Im W_{\leq \lambda}\|_{L^\infty} \| \Im R_{\leq \lambda}\|_{L^\infty} 
\\
\lesssim & \ \sum_{\lambda < 1} 
\lambda^{-1} \sum_{\mu < \lambda} \mu^\frac12 \| \Im W_{\mu}\|_{L^2} 
\sum_{\nu < \lambda} \nu^\frac12 \| \Im R_{\nu}\|_{L^2} 
\\
\lesssim & \  \sum_{\mu,\nu < 1} 
\min\left\{ (\mu/\nu)^\frac12, (\nu/\mu)^\frac12  \right\}
\| \Im W_{\mu}\|_{L^2}  \nu^\frac12 \| \Im R_{\nu}\|_{L^2} 
\\
\lesssim & \   \| \Im W_{<1}\|_{ H_h^\frac14} 
\sum_{\nu < \lambda} \nu^\frac12 \| \Im R_{\leq 1}\|_{H_h^{-\frac14}} 
\\
\lesssim & 
\| \Im W \|_{g^{-\frac{1}{4}}H^{\frac14}_h 
\cap \kappa^{-\frac{1}{4}}H^{\frac34}_h} 
\| \Im R \|_{ g^{\frac{1}{4}}H^{-\frac14}_h+ \kappa^{\frac{1}{4}}H^{-\frac34}_h},
\end{split}
\]
where the last step accounts 
for the rapidly decaying tails in the frequency localizations.

It remains to consider $I_0$, for which 
it suffices to estimate at fixed $\beta \in [0,1]$ 
(the norms for $\Im W$ and $\Im R$ at depth $\beta$ are easily estimated by the similar norms on the top):
\[
\begin{split}
\left|\int_\R J m_x \Im W \Im R \, d\alpha \right| \lesssim
\|J m_x \Im W \|_{g^{-\frac{1}{4}}H^{\frac14}_h \cap \kappa^{-\frac{1}{4}}H^{\frac34}_h} \| \Im R \|_{ g^{\frac{1}{4}}H^{-\frac14}_h+ \kappa^{\frac{1}{4}}H^{-\frac34}_h},
\end{split}
\]
where for the first factor we further estimate 
as in the proof of \eqref{algebra},
\[
\begin{split}
\|J m_x \Im W \|_{g^{-\frac{1}{4}}H^{\frac14}_h \cap \kappa^{-\frac{1}{4}}H^{\frac34}_h} \lesssim & \  \| m_x\|_{W^{1,1}} \|J  \Im W \|_{g^{-\frac{1}{4}}H^{\frac14}_h \cap \kappa^{-\frac{1}{4}}H^{\frac34}_h}
\\ \lesssim & \ (\|J \|_{L^\infty} + (\kappa/g)^\frac14 \| J_\alpha\|_{L^2}) \|\Im W \|_{g^{-\frac{1}{4}}H^{\frac14}_h \cap \kappa^{-\frac{1}{4}}H^{\frac34}_h},
\end{split}
\]
and for the $L^2$ norm of $J_\alpha$ we use our a-priori $X^\kappa$ bound 
given by Theorem~\ref{t:equiv} to get
\[
(\kappa/g)^\frac14 \| J_\alpha\|_{L^2} \lesssim \epsilon.
\]

\subsection{ Proof of Lemma~\ref{l:kappa-diff-re}}
Taking into account the above properties, we bound the integral $I_2^{\kappa}$  by 
\[
|I^\kappa_2| \lesssim   \int_0^T  \iint_{S_{hol}(t)}   |\theta_y| |H_{N}(\theta_{xx}) - \theta_{xx}|\,  d\alpha d\beta
 dt .
\]
As in \cite{AIT}, we split the integration 
region vertically into dyadic pieces, which are contained 
in the regions $A_1(x_0)$, respectively $A_{\lambda}(x_0)$ 
with $h^{-1} < \lambda < 1$ dyadic, and all of which
are contained in $\bfB_{1/h}(x_0)$. We also take 
advantage of the fact that the second factor is smooth on the $h$ scale and
vanishes on the top in order to insert 
a $\beta$ factor. Then we estimate 
\[
\begin{split}
|I^\kappa_2| \lesssim & \   \int_0^T  \iint_{A_1(x_0)} 
 |\beta \theta_y| d\alpha d \beta \sup_{A_1(x_0)}  
 |\beta^{-1} (H_{N}(\theta_{xx}) - \theta_{xx})| +\\
  &  \hspace*{1cm} + \sum_{\lambda = 1}^{h^{-1}}  \lambda^{-1} \sup_{A_\lambda(x_0)} 
 |\beta \theta_y| \sup_{A_\lambda(x_0)}  
 |\beta^{-1} (H_{N}(\theta_{xx}) - \theta_{xx})|\, dt
 \\
 \lesssim &  \left(\! \|\beta \theta_y\|_{L^2_t L^\infty(A_1(x_0))} +\! \sum_{\lambda = 1}^{h^{-1}}  \lambda^{-1}
 \|\beta \theta_y\|_{L^2_t L^\infty(A_\lambda(x_0))}\! \right)\! \| \beta^{-1} (H_{N}(\theta_{xx}) - \theta_{xx}) \|_{L^2_t L^\infty( \bfB^{1/h}(x_0)) }.
\end{split}
\]

Then it suffices to prove the following bounds:
\begin{equation}\label{est-thetab1}
\| \beta \theta_y \|_{L^2_t L^2(A_1(x_0))} \lesssim \| \Im W\|_{LE^{0}},
\end{equation}
\begin{equation}\label{est-thetab}
\| \beta \theta_y \|_{L^2_t L^\infty(A_\lambda(x_0))} \lesssim \| \Im W\|_{LE^{0}},
\end{equation}
 respectively 
\begin{equation}\label{estDN}
\| \beta^{-1} (H_{N}(\theta_{xx}) - \theta_{xx}) \|_{L^2_t L^\infty( \bfB^{1/h}(x_0)) } \lesssim h^{-3} \| \Im W\|_{LE^{0}}.
\end{equation}
Given these three bounds, the conclusion of the Lemma easily follows. It remains to prove \eqref{est-thetab1},
\eqref{est-thetab}, respectively \eqref{estDN}.

\bigskip

{\bf Proof of \eqref{est-thetab1}, \eqref{est-thetab}:}
These bounds are direct consequences of Lemma~\ref{l:theta-LE-hi}, 
respectively Lemma~\ref{l:theta-LE-loc}. 

\bigskip

{\bf Proof of \eqref{estDN}:} Here we are subtracting the Dirichlet
and Neuman extension of a given function. This we already had to do in
\cite{AIT}, where the idea was that the only contributions come from
very low frequencies $\leq h^{-1}$, 
\[
H_N(\theta_{xx}) - \theta_{xx} \approx P_{<1/h} \theta_{xx}.
\]
If instead of $\theta_{xx}$ we had its principal part $\Im W_{\alpha \alpha}$ then the argument would
be identical to \cite{AIT}, gaining two extra $1/h$ factors from the derivatives.
The challenge here is to show that we can bound the very low frequencies of $\theta_{xx}$ 
in a similar fashion. But given the  expression \eqref{second} for $\theta_{xx}$, 
this is also a direct consequence of Lemma~\ref{l:Y}, applied to the function
\[
F(W_\alpha) = \frac{1}{(1+W_\alpha)^2} -1.
\]

\subsection{ Proof of Lemma~\ref{l:kappa-3-re}}
As before we bound  $I_3^{\kappa}$ as
\[
|I^\kappa_3| \lesssim    \int_0^T  \iint_{S_{hol}(t)} | \theta_y |  
|H_N(H(\eta)- \theta_{xx})|  \,  d\alpha d\beta  dt .
\]
We write on the top
\[
\begin{split}
H(\eta) - \theta_{xx} = &\  \Im  \left[   \frac{1}{1+ W_\alpha} \partial_\alpha (J^{-\frac12}(1+ W_\alpha))
 - \frac{1}{1+W_\alpha} \partial_\alpha  \left(\frac{W_\alpha}{1+W_\alpha}\right) \right]
\\
= &\  \Im  \left[  \frac{W_{\alpha \alpha}}{2} 
\left(\frac{1}{(1+ W_\alpha)^\frac32 (1+\bar W_\alpha)^\frac12}  - \frac{2}{(1+ W_\alpha)^3}\right) \right],
\end{split}
\]
where the linear part cancels and we are left with 
a sum of expressions of the form
\[
\Im [ \partial_\alpha F_1(W_\alpha) G_1(\bar W_\alpha) ], \qquad \Im [ \partial_\alpha F_2(W_\alpha)],
\]
where the subscript indicates the minimum degree 
of homogeneity. The first type of expression is the 
worst, as it allows all dyadic frequency interactions, 
whereas the second involves products of 
holomorphic functions which preclude high-high to low 
interactions. Since $F_1(W_\alpha)$  and  $ G_1(\bar W_\alpha) $ 
have the same regularity as $W_\alpha$ and 
$\bar W_\alpha$, to streamline the computation we simply replace 
them by that. Hence we 
end up having to bound trilinear expressions of the form
\begin{equation} \label{I3-main}
I =  \int_0^T  \iint_{S_{hol}(t)}|W_\alpha| |H_N( W_{\alpha \alpha} \bar W_\alpha)|\, d\alpha d\beta dt .
\end{equation}

To estimate this expression we use again a 
dyadic decomposition with respect to depth.
We consider dyadic $\beta$ regions 
$|\beta| \approx \lambda^{-1}$ associated to a frequency 
$\lambda \in (0,h^{-1}]$. But here we need to separate 
into three cases, depending on how $\lambda$ compares to $1$
and also to $\lambda_0 = \sqrt{g/\kappa} > 1$.

\bigskip
{\bf Case 1, $\lambda \geq \lambda_0$.}
The harmonic extension at depth $\lambda$ is a multiplier which 
selects frequencies $\leq \lambda$, with exponentially decaying tails.
Hence we need to estimate an integral of the form 
\begin{equation}\label{I3-high}
I_\lambda = \int_0^T \iint_S 1_{A_1} 1_{ |\beta| \approx \lambda^{-1}} 
|W_{\alpha, <\lambda}| |P_{< \lambda} ( W_{\alpha \alpha} \bar W_\alpha)|\, d\alpha d\beta dt.
\end{equation}
Then we use two local energies for the $W_\alpha$ factors and one 
apriori bound $W_{\alpha \alpha} \in \epsilon (g/\kappa)^\frac14 L^2$ 
arising from the $X_1$ norm plus Bernstein's inequality to bound this by
\[
\begin{split}
I_\lambda \lesssim & \ \lambda^{-1} \| W_\alpha\|_{LE^0} \| P_{< \lambda}
( W_{\alpha \alpha} \bar W_\alpha)\|_{LE^0}
\\
\lesssim & \ \lambda^{-1} \| W_\alpha\|_{LE^0} \lambda^{\frac12} 
\| W_{\alpha \alpha}\|_{L^2} \| W_\alpha\|_{LE^0}
\\
\lesssim & \ \epsilon \lambda^{-1} \lambda^\frac12 
(g/\kappa)^\frac14 \| W_\alpha\|_{LE^{0}}^2,
\end{split}
\]
where the dyadic $\lambda$ summation is trivial for $\lambda > \lambda_0$.

\bigskip

{\bf Case 2, $1 \leq \lambda < \lambda_0$.}
Here we still need to estimate an integral of the form \eqref{I3-high} but we balance norms differently.
Precisely, for the first $W_\alpha$ factor we use the local energy 
norm for $\Im W$ instead, and otherwise 
follow the same steps as before. 
Then we bound the integral in \eqref{I3-high} by
\[
I_\lambda \lesssim \epsilon \lambda^{-1} \lambda \lambda^\frac12 
(g/\kappa)^\frac14 
\|\Im W\|_{LE^{0}}  \| W_\alpha\|_{LE^0}.
\]
Again the dyadic $\lambda$ summation for $\lambda < \lambda_0$ 
is straightforward, and we conclude by the Cauchy-Schwarz inequality.

\bigskip
{\bf Case 3, $\lambda < 1$}. Here the corresponding 
part of the integral \eqref{I3-main} is localized in the intersection of the strip 
$S^\lambda_h(x_0)$ with $A_\lambda$, and we estimate it using $L^\infty$ norms in $A_{\lambda}$,
by 
\[
\begin{split}
I_\lambda = & \  \lambda^{-1} \int_0^T \sup_{A_\lambda(x_0)} |W_\alpha| 
\sup_{A_\lambda(x_0)} |H_N( W_{\alpha \alpha} \bar W_\alpha)|\, dt
\\
\lesssim & \ \lambda^{-1} \| W_{\alpha}\|_{L^2_t L^\infty(A_\lambda)} 
\| H_N ( W_{\alpha \alpha} \bar W_\alpha) \|_{L^2_t L^\infty(A_\lambda)}.
\end{split}
\]

% \[
% \| W_{\alpha,<\lambda}\|_{L^2_t L^\infty(A_\lambda)} 
% \| P_{< \lambda} ( W_{\alpha \alpha} \bar W_\alpha \|_{L^2_t L^\infty(A_\lambda)}.
% \]
For the first factor we use the local energy of $W$,
\[
\| W_{\alpha}\|_{L^2_t L^\infty(A_\lambda)} \lesssim \lambda \| \Im W\|_{LE^{0}}.
\]
For the bilinear factor we recall again that the harmonic extension 
at depth $\beta \approx \lambda^{-1}$ is a multiplier selecting frequencies 
$\leq \lambda$, see \eqref{PN}. Then we use the a-priori $L^2$ 
bound for $W_{\alpha \alpha}$ and the local energy of $W_\alpha$, 
and apply the Bernstein inequality in Lemma \eqref{l:bern-loc}:
\[
\begin{split}
\| P_{< \lambda} ( W_{\alpha \alpha} \bar W_\alpha) \|_{L^2_t L^\infty(B_\lambda)}
\lesssim  & \ \lambda \| W_{\alpha \alpha} \bar W_\alpha \|_{L^2_t L^1_{loc}(B_\lambda)}
\lesssim \lambda \| W_{\alpha \alpha}\|_{L^\infty_t L^2} \| \bar W_\alpha \|_{L^2_t L^2_{loc}(B_\lambda)}
\\
\lesssim & \ \lambda \epsilon (g/\kappa)^\frac14  \lambda^{-\frac12} \| W_\alpha\|_{LE^{0}}.
\end{split}
\]
Overall we obtain the same outcome as in Case 2.

%  \bibliography{bib-AIT}
%  \bibliographystyle{plain}
 
%  \end{document}

\vspace{10mm}

\noindent\textbf{Thomas Alazard}\\
\noindent CNRS and CMLA, \'Ecole Normale Sup{\'e}rieure de Paris-Saclay, Cachan, France

\vspace{3mm}

\noindent\textbf{Mihaela Ifrim}\\
\noindent Department of Mathematics, University of 
Wisconsin, Madison, USA

\vspace{3mm}

\noindent\textbf{Daniel Tataru}\\
\noindent Department of Mathematics, University  of California, Berkeley, USA

 \end{document}